\documentclass[10pt]{amsart}
\usepackage{amssymb}

\newtheorem{proposition}{Proposition}[section]
\newtheorem{lemma}[proposition]{Lemma}
\newtheorem{corollary}[proposition]{Corollary}
\newtheorem{theorem}[proposition]{Theorem}
\newtheorem{question}[proposition]{Question}
\newtheorem{convention}[proposition]{Convention}
\newtheorem{project}[proposition]{Project}
\newtheorem{remark}[proposition]{Remark}
\newtheorem{observation}[proposition]{Observation}

\theoremstyle{definition}
\newtheorem{definition}[proposition]{Definition}

\def\<{\leqslant}
\def\>{\geqslant}
\def\a{\alpha}

\def\e{\varepsilon}

\def\ti{\times}
\def\ot{\otimes}
\def\ra{\rightarrow}

\def\Hom{\mathrm{Hom}}

\DeclareMathOperator{\fpdim}{fpdim}
\DeclareMathOperator{\End}{End}
\DeclareMathOperator{\Ext}{Ext}

\def\fm{\mathfrak{m}}

\def\ann{\mathrm{ann}}
\def\rad{\mathrm{rad}}

\date{}

\begin{document}
\title[Non-semisimple Hopf algebra actions]
{Examples of non-semisimple Hopf algebra \\
actions on Artin-Schelter regular algebras}

\author{Hui-Xiang Chen}
\address{School of Mathematical Sciences, Yangzhou University,
Yangzhou 225002, China}
\email{hxchen@yzu.edu.cn}

\author{Ding-Guo Wang}
\address{School of Mathematical Sciences, Qufu Normal University,
Qufu, Shandong 273165, China}
\email{dgwang@qfnu.edu.cn}

\author{James J. Zhang}
\address{Department of Mathematics, University of Washington,
Box 354350, Seattle, WA 98195, USA}
\email{zhang@math.washington.edu}


\makeatletter
\@namedef{subjclassname@2020}{%
  \textup{2020} Mathematics Subject Classification}
\makeatother

\begin{abstract}
Let $\Bbbk$ be a base field of characteristic $p>0$ and let
$U$ be the restricted enveloping algebra of a 2-dimensional
nonabelian restricted Lie algebra. We classify all inner-faithful
$U$-actions on noetherian Koszul Artin-Schelter regular algebras
of global dimension up to three.
\end{abstract}
\maketitle

\setcounter{section}{-1}
\section{Introduction}
\label{xxsec0}

Invariant theory of commutative polynomial rings under finite group
actions is closely connected to commutative algebra and algebraic
geometry. Artin-Schelter regular algebras \cite{AS}, viewed as a
natural noncommutative generalization of the commutative polynomial
rings, play an important role in noncommutative algebraic geometry,
representation theory, and the study of noncommutative algebras
\cite{ATV1,ATV2, CV}. Hopf actions (including group actions)
on Artin-Schelter regular algebras have been studied extensively by
many authors in recent years, see \cite{CKWZ1, CKWZ2, CKWZ3, CG,
FKMW1, FKMW2, FKMP, KKZ1, KKZ2, KWZ, KZ} and so on. A very nice survey was
given by Kirkman \cite{Ki} a few years ago. In most papers, only
semisimple Hopf algebras are considered due to the fact that
non-semisimple Hopf actions are much more difficult to handle.
A list of significant differences between semisimple and non-semisimple
actions can be found in Observation \ref{xxobs4.1}.

Recall that a Hopf $H$-action on an algebra $A$ is called
{\it inner-faithful} if there is no nonzero Hopf ideal $I\subseteq H$
such that $I A=0$ \cite[Definition 1.5]{CKWZ1}. Our goal is to
construct examples of inner-faithful and homogeneous
$U$-actions on $T$ where $U$ is the non-semisimple Hopf algebra
given in Definition \ref{xxdef0.1} and $T$ is a connected
graded Artin-Schelter regular algebra. The main result consists
of Proposition \ref{xxpro3.2} and Theorem \ref{xxthm0.5} that together
classify all inner-faithful $U$-actions on noetherian Koszul Artin-Schelter
regular algebras of global dimension at most three.

Throughout let $\Bbbk$ be a base field with ${\rm char} \; \Bbbk=p>0$.

\begin{definition}
\label{xxdef0.1}
Let $U$ be the $\Bbbk$-algebra generated by $u$ and $w$ and
subject to the relations
\begin{equation}
\label{E0.1.1}\tag{E0.1.1}
u^p=0, \ w^p=w, \ [w,u](:=wu-uw)=u.
\end{equation}
Then $U$ has a Hopf algebra structure with coalgebra
structure and antipode determined by
$$\begin{array}{lll}
\Delta(u)=u\ot 1+1\ot u, & \e(u)=0, & S(u)=-u,\\
\Delta(w)=w\ot 1+1\ot w, & \e(w)=0, & S(w)=-w.\\
\end{array}$$
Note that $\dim_{\Bbbk} U=p^2$ and $\{u^iw^j|0\leqslant i, j
\leqslant p-1\}$ is a $\Bbbk$-basis of $U$. It is easy to
see that $U$ is isomorphic, as a Hopf algebra, to the
restricted enveloping algebra of the 2-dimensional nonabelian
restricted Lie algebra ${\mathfrak g}:=\Bbbk u\oplus \Bbbk w$
with structure determined by \eqref{E0.1.1}.
\end{definition}

A very first step of understanding $U$-actions on Artin-Schelter
regular algebras is to work out representations of $U$. Similar
to the Taft algebras, there are exactly $p^2$ indecomposable
$U$-modules up to isomorphisms, denoted by
$$\{M(l, i)|1\leqslant l\leqslant p, i\in\mathbb{Z}_p:={\mathbb Z}/(p)\}$$
where $\dim_{\Bbbk} M(l,i)=l$ for all $l,i$, see Convention
\ref{xxcon1.4} and Proposition \ref{xxpro1.7}. We also need
the following tensor decomposition result.

\begin{theorem}
\label{xxthm0.2} Retain the notation as above.
Let $r, r'\in{\mathbb Z}_p$.
\begin{enumerate}
\item[(1)] 
Let $1\leqslant l\<m\< p$ and $l+m\leqslant p$. Then
$$M(l,r)\ot M(m,r')\cong\bigoplus\limits_{i=1}^l
M(m-l-1+2i, r+r'+l-i).$$
\item[(2)] 
Let $1\leqslant l\<m\<p$ and $l+m>p$. Then
$$M(l,r)\ot M(m,r')\cong(\bigoplus\limits_{i=1}^{p-m}
M(m-l-1+2i, r+r'+l-i))
\bigoplus(\bigoplus\limits_{i=1}^{l+m-p}
M(p, r+r'+i-1)).$$
\end{enumerate}
\end{theorem}

The proof of Theorem \ref{xxthm0.2} follows from the ideas in \cite{CVZ}.
Note that the Green ring (or the representation ring) of Hopf algebras
has been studied extensively, see \cite{Ch2, CVZ, HY, LH, LZ, SY, WLZ1,
WLZ2} and more. If we can describe the Green ring of a Hopf algebra $H$,
then it is extremely useful for understanding the representations, the
fusion rules, Grothendieck group, Frobenius-Perron dimension
\cite{LSYZ, Xu, ZZ}, and many other invariants and structures of $H$.
However, it is notoriously difficult to understand the Green ring for
a general Hopf algebra (even, of small dimension, see Question \ref{xxque5.1}).
It is not a surprise that most of positive results so far concern Hopf
algebras of finite or tame representation type. Using Theorem
\ref{xxthm0.2} we can present the Green ring of $U$ -- another non-semisimple
Hopf algebra of finite representation type [Corollary \ref{xxcor0.3}].

For each positive integer $n$, define
$$f_{n}(y, z)=\sum\limits_{i=0}^{[(n-1)/2]}
(-1)^i\binom{n-1-i}{i}y^iz^{n-1-2i},$$
where $[(n-1)/2]$ denotes the integer part of $(n-1)/2$.
Recall that $p={\text{char}}\; \Bbbk$. Let $I$ denote the
ideal of ${\mathbb Z}[y,z]$ generated by $y^p-1$ and
$(z-y-1)f_p(y,z)$.

\begin{corollary}[Corollary \ref{xxcor2.15}]
\label{xxcor0.3}
The Green ring of $U$ is isomorphic to the
factor ring ${\mathbb Z}[y,z]/I$.
\end{corollary}

The proof of this corollary is based on the tensor
decomposition of indecomposable $U$-modules
given in Theorem \ref{xxthm0.2}. Corollary \ref{xxcor0.3}
should be compared with \cite[Theorem 3.10]{CVZ}.
One basic message of Theorem \ref{xxthm0.2} and Corollary
\ref{xxcor0.3} is that the representation theory of $U$ is
similar to the representation theory of the Taft algebras,
though $U$ does not have any nontrivial grouplike elements.

Another application of Theorem \ref{xxthm0.2} is computation
of the Frobenius-Perron dimension of $U$-modules, which was studied
in \cite{Xu}. We give some comments in the last section, see
Remark \ref{xxrem5.5}.

Going back to our main topic, note that the explicit description of
the tensor of two representations of $U$ [Theorem \ref{xxthm0.2}]
is the key to understanding $U$-actions on Artin-Schelter regular algebras.
If $T$ is a noetherian Artin-Schelter regular of global dimension two,
then the existence of an inner-faithful $U$-action on $T$ forces $T$
to be a commutative polynomial ring. From now on Artin-Schelter is
abbreviated as AS.

\begin{proposition}
\label{xxpro0.4}
Let $H$ be any Hopf algebra containing $U$ as a Hopf subalgebra. If
$H$ acts inner-faithfully and homogeneously on a noetherian Koszul 
AS regular algebra $T$ of global dimension two, then $T$ is isomorphic 
to the commutative polynomial ring $\Bbbk[x_1,x_2]$.
\end{proposition}

Explicit $U$-actions on a noetherian Koszul AS regular algebra of global
dimension two are given in Proposition \ref{xxpro3.2}. The next result
is a classification of all inner-faithful $U$-actions on $T$ when $T$
has global dimension 3. Historically, over an algebraically closed
field of characteristic zero, AS regular algebras of
global dimension three were classified by Artin, Schelter, Tate and
Van den Bergh in their seminal papers \cite{AS, ATV1, ATV2}. Since
our base field has positive characteristic, we have to use a
different method. But the purpose of this paper is not to classify 
all AS regular algebras over a field of positive characteristic 
which is another extremely difficult project, see Remark 
\ref{xxrem0.7}(2).

\begin{theorem}
\label{xxthm0.5}
Let $\Bbbk$ be an algebraically closed field of characteristic $p>0$.
Let $T$ be a noetherian Koszul AS regular $\Bbbk$-algebra of global
dimension three and let $V=T_1$. Suppose there is an inner-faithful
and homogeneous $U$-action on $T$. Then one of the following occurs,
up to a change of basis.
\begin{enumerate}
\item[(1)]
$T$ is the commutative polynomial ring $\Bbbk[V]$
where the left $U$-module $V$ is either $M(3,i)$ or
$M(2,i)\oplus M(1,j)$ for some $i,j\in {\mathbb Z}_p$.
\item[(2)] 
$p=3$, $V=M(3,i)=\Bbbk x_1\oplus \Bbbk x_2\oplus \Bbbk x_3$ as
in Convention \ref{xxcon1.4}, and the relations of $T$ are
$$x_2x_1-x_1x_2+x_3^2=0, \quad
x_3x_1-x_1x_3=0, \quad x_3x_2-x_2x_3=0.$$
\end{enumerate}
In the rest of the theorem let $V=M(2,i)\oplus M(1,j)=
(\Bbbk x_1\oplus \Bbbk x_2)\oplus \Bbbk y$ for some
$i,j\in {\mathbb Z}_p$.
\begin{enumerate}
\item[(3)] 
$j=i+1$ and the relations of $T$ are
$$
x_1x_2 +  x_1y- yx_1=0,\quad
x_2x_1 -x_1x_2=0,\quad x_2^2 +   x_2y- yx_2=0.
$$
\item[(4)] 
$(i-j)(2i+1-2j)\neq 0$ in ${\mathbb Z}_p$, and the relations
in $T$ are
$$yx_1+ a x_1 y=0, \quad
yx_2+a x_2 y=0, \quad x_2x_1-x_1x_2=0$$
for $a\neq 0$ in $\Bbbk$.
\item[(5)] 
$2i+1-2j=0$ in
${\mathbb Z}_p$, and the relations in $T$ are
$$yx_1+ a x_1 y=0, \quad
yx_2+a x_2 y=0, \quad
x_2x_1-x_1x_2+ \epsilon y^2=0$$
where $a\neq 0$, $\epsilon=0$ or $1$ and $\epsilon (a^2-1)=0$.
\item[(6)] 
$i=j$ in ${\mathbb Z}_p$,
and the relations in $T$ are
$$yx_1+ a x_1 y+ b y^2=0, \quad
yx_2+a x_2 y=0,\quad
x_2x_1-x_1x_2+ \epsilon x_2 y=0$$
where $a\neq 0$, $\epsilon=0$ or $1$ and  $(a+1)(b-\epsilon)=0$.
\item[(7)] 
$j=i+2$ in ${\mathbb Z}_p$ and the relations in $T$ are
$$
y x_2- x_2 y=0,\quad
x_1 x_2- x_2 x_1 +c x_2 y+ b y^2 =0,\quad
x_2^2 + y x_1-x_1 y +d y^2=0
$$
where $c\neq 0$ or $d\neq 0$ only if $p=2$ and
where $b \neq 0$ only if $p=3$.
\item[(8)] 
$p=2$, $i=j$ in ${\mathbb Z}_2$ and the relations in $T$ are
$$
yx_2+ x_2y=0,\quad
x_1^2 + y^2+e x_2^2=0,\quad
x_1 x_2+x_2 x_1=0,
$$
where $e\in \Bbbk$.
\item[(9)] 
$p=2$, $i\neq j$ in ${\mathbb Z}_2$ and the relations in $T$ are
$$
y x_2+ x_2 y +b y^2=0,\quad
x_1^2 +c x_2 y+ y^2+e x_2^2,\quad
x_1 x_2+x_2 x_1=0,
$$
with $e\in \Bbbk$ and $(b,c)=(0,1)$ or $(1,0)$.
\item[(10)] 
$p=2$, $i\neq j$ in ${\mathbb Z}_2$, and the relations of $T$
are one of the following forms:
\begin{enumerate}
\item[(10a)]
$c\in \Bbbk$, and
$$
x_1^2+c (x_2y+yx_2)+y^2=0,
x_1x_2+x_2x_1=0,
x_2^2+x_2y+yx_2=0.
$$
\item[(10b)]
$e\in \Bbbk$, and
$$
x_1^2+ x_2y+yx_2+e y^2=0,\quad
x_1x_2+x_2x_1=0,\quad
x_2^2+y^2=0.
$$
\end{enumerate}
\end{enumerate}
\end{theorem}

Combining Theorem \ref{xxthm0.5} with Lemma \ref{xxlem4.2}(3),
we obtain

\begin{corollary}
\label{xxcor0.6}
Let $H$ be any Hopf algebra containing $U$ as a Hopf subalgebra. If 
$H$ acts inner-faithfully and homogeneously on a noetherian Koszul 
AS regular algebra $T$ of global dimension three, then $T$ is one 
of the AS regular algebras listed in Theorem {\rm{\ref{xxthm0.5}}}.
\end{corollary}

\begin{remark}
\label{xxrem0.7} 
The following remarks aim to clarify potential confusion.
\begin{enumerate}
\item[(1)]
The list in Theorem {\rm{\ref{xxthm0.5}}} is long. This is due
to the fact that there are many different AS regular algebras
of global dimension three. Even for the same $T$ there could be
different and non-equivalent $U$-actions on $T$.
\item[(2)]
The classification of all noetherian Koszul AS regular algebras
of global dimension three over a field of positive characteristic
has not been done. This could be a huge project which is parallel
to the work of Artin, Schelter, Tate and Van den Bergh
\cite{AS, ATV1, ATV2}. The AS regular algebras listed in Theorem
\ref{xxthm0.5} form a very small portion in the class of all
noetherian Koszul AS regular algebras of global dimension three.
\item[(3)]
If we are given a Koszul AS regular algebra $T$ generically,
it is likely that there is no inner-faithful $U$-action on $T$.
For example, let $T$ be a skew polynomial ring
$\Bbbk_{p_{ij}}[x_1,x_2,x_3]$ where $p_{ij}\neq 1$ for all $i<j$,
then there is no inner-faithful $U$-action on $T$
{\rm{[Proposition \ref{xxpro4.3}]}}.
\item[(4)]
There are some obvious overlaps between part {\rm{(1)}} and parts
{\rm{(4,5,6)}} in Theorem {\rm{\ref{xxthm0.5}}}.
\item[(5)]
Given a specific Hopf algebra $H$ strictly containing $U$, there could 
be one or more AS regular algebras $T$ listed in Theorem \ref{xxthm0.5}
on which $H$ cannot act inner-faithfully. 
\item[(6)]
Given a specific Hopf algebra $H$ containing $U$, there could 
be more than one $H$-action on the same $T$, similar to 
some parts of Theorem {\rm{\ref{xxthm0.5}}}.
\item[(7)]
In a weak sense, Corollary {\rm{\ref{xxcor0.6}}} provides a 
``universal'' classification of $H$-actions on noetherian Koszul 
AS regular algebras of global dimension three for all $H$ 
containing $U$. See Observation {\rm{\ref{xxobs4.5}}} for more comments.
\end{enumerate}
\end{remark}

There are a few reasons for us to consider to this particular 
Hopf algebra $U$.

The first one is that $U$ is of finite representation type. This
makes it possible to list all $U$-module $V:=T_1$, which serves as
an initial step in our classification. If $U$ were of wild
representation type, it is unrealistic to list all $U$-modules
(even for a given dimension).

The second reason is that $U$ is generated by primitive
elements $u$ and $w$. So $u$ and $w$ acts on an algebra $T$
as derivatives. This kind of well-understood operation is
helpful when we are dealing with a lot of computation.

The first two reasons make the project possible. The third reason
is our motivation, namely, $U$ is not semisimple and does not
contain any nontrivial grouplike elements. The invariant theory
under $U$-action is different from the classical invariant theory of
polynomial rings under finite group actions. Observation
\ref{xxobs4.1} lists some significant differences in terms
of homological properties. We would like to use the examples
in this paper to further study non-semisimple Hopf actions on
AS regular algebras.

Classifying $U$-actions on AS regular algebras is also helpful
for understanding other $H$-actions on AS regular
algebras when $H$ is related to $U$, see Corollary \ref{xxcor0.6},
Remark \ref{xxrem0.7} and Observation \ref{xxobs4.5}.

The last reason is that $U$ appears in several different
topics of recent interest. The Hopf algebra $U$ originates from
Lie theory, and is related to (small) quantum group theory,
computation of Frobenius-Perron dimension, study of the Azumaya locus
of a family of PI Hopf algebras, called
{\it iterated Hopf Ore extensions}, or {\it IHOEs}, as we explain
next. In \cite{BZ}, a family of noncommutative PI Hopf algebras $H$ in
characteristic $p$ were studied. The algebra $U$ appears
naturally as a ``fiber'' at every non-Azumaya point of 2-step IHOEs
in \cite[Proposition 8.2(3)]{BZ}. The only other possible fiber of
other 2-step IHOEs in \cite{BZ} is the Hopf algebra
$U_0:=\Bbbk[X,Y]/(X^p,Y^p)$, -- the restricted
enveloping algebra of the abelian Lie algebra of dimension 2 with
trivial restriction -- see \cite[Proposition 8.2(2)]{BZ}. Geometric
and representation theoretic properties of the Hopf algebras in
\cite{BZ} are largely encoded in the properties of algebras $U$
and $U_0$. Using the representations of $U$, we can describe
all brick modules over 2-step IHOEs, see Remark \ref{xxrem5.6}.

This paper is organized as follows. Sections \ref{xxsec1}
and \ref{xxsec2} follow from the structure of \cite{CVZ} and
prove Theorem \ref{xxthm0.2} and Corollary \ref{xxcor0.3}. In Section
\ref{xxsec3} we classify all $U$-actions on noetherian Koszul
AS regular algebras of global dimension at most three as stated
in Proposition \ref{xxpro3.2} and Theorem \ref{xxthm0.5}.
We give some easy, but interesting, observations in Section \ref{xxsec4}.
Section \ref{xxsec5} contains some comments, projects, 
and remarks.

\section{Preliminaries and Representations of $U$}
\label{xxsec1}

Unless otherwise stated, all algebras,
Hopf algebras and modules are defined over $\Bbbk$. All modules
are left modules and all maps are $\Bbbk$-linear.
We use $\otimes$ for $\otimes_{\Bbbk}$. For the theory of Hopf algebras,
we refer to the standard text books \cite{Ka, Ma, Mo, Sw}.
Let $\Bbbk^{\ti}$ denote the multiplicative group of all nonzero elements
in the field $\Bbbk$.

It is well-known that $\mathbb{Z}_p$ is a subfield of $\Bbbk$. For an
integer $r\in{\mathbb Z}$, the image of $r$ under the canonical
epimorphism $\mathbb {Z}\ra\mathbb {Z}_p$ is still denoted by $r$.
Then $\mathbb{Z}_p^{\times}=\{1, 2, \cdots, p-1\}$, which is a cyclic
subgroup of the multiplicative group $\Bbbk^{\times}$.

Let $H$ be a finite dimensional Hopf algebra. The {\it representation
rings} (or the {\it Green rings}) $r(H)$ and $R(H)$ can be defined as
follows. Recall that $r(H)$ is the abelian group generated by the
isomorphism classes $[V]$ of finitely generated $H$-modules $V$ modulo
the relations $[M\oplus V]=[M]+[V]$. The multiplication of $r(H)$ is
induced by the tensor product of $H$-modules, that is, $[M][V]=[M\ot V]$.
Then $r(H)$ is an associative ring. Recall that $R(H)$ is an associative
$\Bbbk$-algebra defined by $\Bbbk\ot_{\mathbb Z}r(H)$. Note that $r(H)$
is a free abelian group with the $\mathbb Z$-basis
$\{[V]|V\in{\rm ind}(H)\}$, where ${\rm ind}(H)$ denotes the category of
all finitely generated indecomposable $H$-modules.

For a module $M$ over a finite dimensional algebra, let ${\rm rl}(M)$
denote the Loewy length (=radical length=socle length) of $M$, and
${\rm l}(M)$ denote the length of $M$. Let $P(M)$ denote the
projective cover of $M$ and $I(M)$ denote the injective hull of $M$.

Let $U$ be defined as in Definition \ref{xxdef0.1}. The following
facts about $U$ are folklore. Let $B$ denote
the subalgebra of $U$ generated by $w$. Then $B$ is a $p$-dimensional
semisimple Hopf subalgebra of $U$. Moreover, there is a Hopf algebra
epimorphism $\pi: U\ra B$ defined by $\pi(u)=0$ and $\pi(w)=w$. It
follows that $\ker \pi=(u)\supseteq J(U)$, the Jacobson radical of $U$.
On the other hand, since $Uu=uU$ and $u^p=0$, $J(U)\supseteq (u)=uU$,
the ideal of $U$ generated by the normal element $u$. Hence $\ker \pi
=(u)=J(U)$. Thus, an $U$-module $M$ is semisimple if and only if
$u \cdot M=0$, and moreover, $M$ is simple if and only if $u \cdot M=0$
and $M$ is simple as a module over $B$. Note that
$w^p-w=\prod_{i \in {\mathbb Z}_p} (w-i)$ over a field of characteristic
$p$. Therefore, we have the following lemma, and its proof follows
from the fact that $w$ acts semisimply on any finite dimensional
$U$-module.

\begin{lemma}
\label{xxlem1.1}
For every finite dimensional $U$-module $M$ and every $i \in
{\mathbb Z}_p\subseteq \Bbbk$, let
$M[i]=\{m\in M| w \cdot m=i m\}$, then we have
$$M=\bigoplus_{i \in{\mathbb Z}_p}M[i],\quad
\text{and} \quad uM[i]\subseteq M[i+1].$$
\end{lemma}

\begin{lemma}
\label{xxlem1.2}
There are $p$ non-isomorphic simple $U$-modules
$\{S_i\}_{i\in {\mathbb Z}_p}$, and each $S_i$
is 1-dimensional and determined by
$$u\cdot x=0,\quad {\text{and}}\quad w\cdot x=i x,$$
where $x$ is a basis element in $S_i$.
\end{lemma}

Note that $J(U)^m=u^mU$ for all $m\geqslant 1$. Hence
$J(U)^{p-1}\neq 0$ and $J(U)^p=0$. This means that
the Loewy length of $U$ is $p$. Since every simple
$U$-module is 1-dimensional, ${\rm l}(M)={\rm dim}(M)$
for all $U$-modules $M$. Let $M$ be an $U$-module.
Since $J(U)^s=Uu^s=u^sU$, we have $\rad^s(M)
=u^s\cdot M$ for all $s\geqslant 1$.

\begin{lemma}
\label{xxlem1.3}
Let $1\leqslant l\leqslant p$ and $i\in \mathbb{Z}_p$.
Then there is an algebra homomorphism $\rho_{l,i}: U\ra M_l(k)$
given by
$$\rho_{l,i}(u)=\left(
\begin{array}{ccccc}
0 &  &  &  &  \\
1 & 0 &  &  &  \\
& 1 & \ddots &  & \\
&   & \ddots & \ddots &  \\
&   &   &1 & 0 \\
\end{array}
\right),\
\rho_{l,i}(w)=\left(
                  \begin{array}{ccccc}
                    i&&&&\\
                    &i+1&&&\\
                    &&i+2&&\\
                    &&&\ddots&\\
                    &&&&i+l-1\\
                  \end{array}
                \right).$$
Let $M(l,i)$ denote the corresponding left $U$-module.
\end{lemma}

\begin{proof}
It follows from a straightforward verification.
\end{proof}

In some other papers when one uses a different convention,
the matrix $\rho_{l,i}(u)$ in Lemma \ref{xxlem1.3} should 
be replaced by its transpose. But we fix the following
convention.

\begin{convention}
\label{xxcon1.4}
By Lemma {\rm{\ref{xxlem1.3}}}, the module $M(l, i)$ has
a $\Bbbk$-basis $\{x_1, x_2, \cdots, x_l\}$ such that
$w\cdot x_j=(i+j-1)x_j$ for all $1\leqslant j\leqslant l$
and
$$u\cdot x_j=\left\{\begin{array}{ll}
x_{j+1},& 1\leqslant j\leqslant l-1,\\
0,& j=l.\\
\end{array}\right.
$$
Hence we have $x_j=u^{j-1}\cdot x_1$ for all
$2\leqslant j\leqslant l$. Such a basis is called a
{\it standard basis} of $M(l, i)$.
\end{convention}

We now list some easy facts.

\begin{lemma}
\label{xxlem1.5}
The following hold.
\begin{enumerate}
\item[(1)]
${\rm soc}(M(l,i))=kx_l\cong S_{i+l-1}$ and
$M(l,i)/{\rm rad}(M(l,i))\cong S_i$.
\item[(2)]
$M(l, i)$ is indecomposable and uniserial.
\item[(3)]
If $1\leqslant l'\leqslant p$ and $i'\in\mathbb{Z}_p$, then
$M(l, i)\cong M(l', i')$ if and only if $l'=l$ and $i'=i$.
\end{enumerate}
\end{lemma}

\begin{proof}
It is similar to the proof of \cite[Lemma 2.3]{CVZ}.
\end{proof}

The next corollary is similar to \cite[Corollary 2.4]{CVZ}.

\begin{corollary}
\label{xxcor1.6}
The following hold.
\begin{enumerate}
\item[(1)]
$M(l,i)$ is simple if and only if $l=1$. In this case,
$M(1, i)\cong S_i$.
\item[(2)]
$M(l,i)$ is projective {\rm{(}}respectively, injective{\rm{)}}
if and only if $l=p$.
\item[(3)]
$M(p,i)\cong P(S_i)\cong I(S_{i-1})$.
\end{enumerate}
\end{corollary}

\begin{proof}
(1) It follows from Lemma \ref{xxlem1.5}(1).

(2,3) Note that every finite dimensional Hopf algebra is
self-injective as an algebra. If $l=p$, then it follows
from \cite[Lemma 3.5]{Ch3} that $M(p,i)$ is projective and
injective.

Define $e_0=1-w^{p-1}$ (for $i=0$) and
$e_i=\frac{1}{p-1}\sum_{j=1}^{p-1}i^{-j}w^j$ for
$i\in\mathbb{Z}_p^{\times}$. It is easy to check that
$e_0^2=e_0$, $e_0e_i=0$ and $e_ie_l=\delta_{il}e_l$ for
all $i,l\in\mathbb{Z}_p^{\times}$. That is,
$\{e_0, e_1, \cdots, e_{p-1}\}$ is a set of orthogonal
idempotents of $U$. Since $\mathbb{Z}_p^{\times}$ is a
cyclic group of order $p-1$,
$\sum_{i=1}^{p-1}i^{-j}=\delta_{j,p-1}(p-1)$ for any
$1\<j\<p-1$. Then one can check that
$\sum_{i=1}^{p-1}e_i=w^{p-1}$, and so $\sum_{i=0}^{p-1}e_i=1$. We also have
$we_i=ie_i$ and $u^{p-1}e_i\neq 0$, where the latter follows
from the fact that $\{u^i w^j\mid 0\< i,j \< p-1\}$ is a
$\Bbbk$-linear basis of $U$. Therefore,
$$Ue_i={\rm span}\{e_i, ue_i, \cdots, u^{p-1}e_i\}\cong M(p,i).$$
Thus, we have the decomposition of the regular module $U$ as follows
$$U=\bigoplus\limits_{i=0}^{p-1}Ue_i\cong
\bigoplus\limits_{i=0}^{p-1}M(p, i).$$
Hence $M(p,i)\cong P(S_i)$, and $M(p, 0)$, $M(p, 1)$, $\cdots$,
$M(p, p-1)$ are all non-isomorphic indecomposable projective
$U$-modules. So parts (2,3) follow from Lemma \ref{xxlem1.5}.
\end{proof}

Since the indecomposable projective $U$-modules are uniserial,
any indecomposable $U$-module is uniserial and is isomorphic to
a quotient of an indecomposable projective module. Therefore,
we have the following proposition (which is well-known).

\begin{proposition}
\label{xxpro1.7}
Up to isomorphism, there are $p^2$ finite dimensional
indecomposable $U$-modules as follows
$$\{M(l, i)|1\leqslant l\leqslant p, i\in\mathbb{Z}_p\}.$$
\end{proposition}

It is easy to see that $\Hom_U(M(l,i),M(l,i))=\Bbbk$
for all modules in Proposition \ref{xxpro1.7}.

\section{Tensor decomposition and the Green ring of $U$}
\label{xxsec2}

Since $U$ is cocommutative, the tensor category $U$-mod is
symmetric. By Proposition \ref{xxpro1.7}, there are $p^2$
non-isomorphic indecomposable modules over $U$, namely,
$$\{M(l,r)|1\leqslant l\leqslant p, r\in\mathbb{Z}_p\}.$$

For any $U$-module $M$ and $r\in\mathbb{Z}_p$, recall from
Lemma \ref{xxlem1.1} that
$$M[r]=\{m\in M|w\cdot m=rm\}.$$
Then $M[r]$ is a subspace of $M$. Next we list
some easy facts in the following lemma.

\begin{lemma}
\label{xxlem2.1}
Let $M$ be an $U$-module. Then
\begin{enumerate}
\item[(1)]
If $M$ is indecomposable, then ${\rm dim}(M[r])\leqslant 1$
for every $r\in\mathbb{Z}_p$.
\item[(2)]
If $M=M_1\oplus M_2$ for some submodules $M_1$ and $M_2$,
then $M[r]=M_1[r]\oplus M_2[r]$ for every $r\in\mathbb{Z}_p$.
\item[(3)]
When $M$ is decomposed into a direct sum of indecomposable
submodules, the number of summands is at least
${\rm max}\{{\rm dim}(M[r])|r\in\mathbb{Z}_p\}$.
\end{enumerate}
\end{lemma}

\begin{lemma}
\label{xxlem2.2}
Let $1\leqslant l\leqslant p$ and $r, r'\in\mathbb{Z}_p$. Then
$$S_{r'}\ot M(l,r)\cong M(l,r)\ot S_{r'}\cong M(l, r+r')$$
as $U$-modules. In particular, $S_r\otimes S_{r'}\cong S_{r+r'}$
and $M(l, r)\cong S_r\ot M(l,0)$.
\end{lemma}

Throughout the rest of the section, let $2\leqslant l\<m\leqslant p$,
and let
$$M=M(l, 0)\ot M(m, 0).$$
Let $\{x_1, x_2, \cdots, x_l\}$ and $\{y_1, y_2, \cdots, y_m\}$
be the standard bases of $M(l, 0)$ and $M(m, 0)$, respectively,
as stated in Convention \ref{xxcon1.4}. Then
$$\{x_i\ot y_j|1\leqslant i\leqslant l, 1\leqslant j\leqslant m\}$$
is a $\Bbbk$-basis of $M$. For any $2\leqslant s\leqslant l+m$, let
$$M(s)={\rm span}\{x_i\ot y_j|i+j=s\}.$$
Then we have $M=\bigoplus\limits_{s=2}^{l+m}M(s)$ as $\Bbbk$-spaces.

\begin{lemma}
\label{xxlem2.3}
Retain the above notation.
\begin{enumerate}
\item[(1)]
$u\cdot M(s)\subseteq M(s+1)$ for all $2\leqslant s\leqslant l+m$,
where $M(l+m+1)=0$.
\item[(2)]
$M(s)\subseteq M[s-2]$ for all $2\leqslant s\leqslant l+m$.
\item[(3)]
${\rm dim}(M(s))=\left\{\begin{array}{ll}
s-1,& if\ 2\leqslant s\leqslant l+1\\
l,& if\ l+1<s<m+1\\
l+m+1-s,& if\ m+1\leqslant s\leqslant l+m\\
\end{array}\right.$
\end{enumerate}
\end{lemma}

\begin{proof}
It follows from a straightforward verification.
\end{proof}

\begin{lemma}
\label{xxlem2.4}
The socle of $M$ has the following decomposition
$${\rm soc}(M)=\bigoplus\limits_{2\leqslant s\leqslant l+m}
{\rm soc}(M)\cap M(s).$$
\end{lemma}

\begin{proof}
It follows from Lemma \ref{xxlem2.3}(1) since ${\rm soc}(M)=\{z\in M|u\cdot z=0\}$.
\end{proof}

\begin{lemma}
\label{xxlem2.5}
The following statements hold.
\begin{enumerate}
\item[(1)]
If $2\leqslant s\leqslant m$, then ${\rm soc}(M)\cap M(s)=0$.
\item[(2)]
If $m+1\leqslant s\leqslant l+m$, then
${\rm dim}({\rm soc}(M)\cap M(s))=1$.
\item[(3)]
${\rm dim}({\rm soc}(M))=l$.
\end{enumerate}
\end{lemma}

\begin{proof}
(1) Let $2\leqslant s\leqslant l$ and
let $z\in M(s)$. Then $z=\sum\limits_{i=1}^{s-1}\a_ix_i\ot y_{s-i}$
for some $\a_i\in \Bbbk$. A straightforward computation shows that
$$u\cdot z=\a_1x_1\ot y_{s}
+\sum\limits_{2\<i\<s-1}(\a_{i-1}+\a_i)x_{i}\ot y_{s+1-i}
+\a_{s-1}x_{s}\ot y_1.$$
Now by an easy linear algebra argument, $u\cdot z=0$ if and
only if $z=0$. Thus, $z\in{\rm soc}(M)$ if and only if $z=0$.
This shows that ${\rm soc}(M)\cap M(s)=0$ for all
$2\leqslant s\leqslant l$.

Now let $l+1\leqslant s\leqslant m$ and let $z\in M(s)$.
In this case, $l<m$ and
$z=\sum\limits_{i=1}^l\a_ix_i\ot y_{s-i}$ for some
$\a_i\in \Bbbk$. Hence we have
$u\cdot z=\a_1x_1\ot y_{s}
+\sum\limits_{i=2}^{l}(\a_{i-1}+\a_i)x_{i}\ot y_{s+1-i}$.
Thus, by a similar argument as above, one can show that
$z\in{\rm soc}(M)$ if and only if  $z=0$.
Hence ${\rm soc}(M)\cap M(s)=0$ for all
$l+1\leqslant s\leqslant m$.

(2) Obviously, $u\cdot M(l+m)=0$. Hence
$M(l+m)\subseteq{\rm soc}(M)$,
and so $M(l+m)\cap{\rm soc}(M)=M(l+m)$ is one dimensional.
Now let $m+1\leqslant s< l+m$
and $z\in M(s)$. Then
$z=\sum\limits_{i=s-m}^{l}\a_ix_i\ot y_{s-i}$
for some $\a_i\in \Bbbk$. One can check that
$$u\cdot z=\sum\limits_{i=s+1-m}^{l}(\a_{i-1}+\a_i)x_{i}\ot y_{s+1-i}.$$
Thus, $z\in{\rm soc}(M)$ if and only if $u\cdot z=0$
if and only if $\a_{i-1}+\a_i=0$ for all
$s+1-m\leqslant i\leqslant l$. It follows that
${\rm dim}({\rm soc}(M)\cap M(s))=1$ in this case.

(3) It follows from (1), (2) and Lemma \ref{xxlem2.4}.
\end{proof}

\begin{corollary}
\label{xxcor2.6}
Retain the above notation.
\begin{enumerate}
\item[(1)]
For any $m+1\leqslant s\leqslant l+m$, let
$z_s=\sum\limits_{i=s-m}^{l}(-1)^ix_i\ot y_{s-i}$.
Then ${\rm soc}(M)\cap M(s)=\Bbbk z_s$.
\item[(2)]
${\rm soc}(M)={\rm span}\{z_s| m+1\leqslant s\leqslant l+m\}$.
\item[(3)]
${\rm soc}(M)\cong\bigoplus\limits_{s=m+1}^{l+m}
M(1, s-2)\cong\bigoplus\limits_{s=m+1}^{l+m}S_{s-2}$.
\end{enumerate}
\end{corollary}

\begin{proof}
It follows from the proof of Lemma \ref{xxlem2.5}.
\end{proof}

Define ${\rm head}(M)=M/u\cdot M$.

\begin{corollary}
\label{xxcor2.7}
Retain the above notation. Then
${\rm head}(M)\cong \bigoplus\limits_{i=0}^{l-1}M(1, i)
\cong \bigoplus\limits_{i=0}^{l-1}S_i$.
\end{corollary}

\begin{proof}
By Lemma \ref{xxlem2.3}(1), we have $u\cdot M(s)
\subseteq M(s+1)$, and hence
$$u\cdot M=\bigoplus\limits_{s=2}^{l+m-1}u\cdot M(s).$$
Now by Lemmas \ref{xxlem2.3} and \ref{xxlem2.5} and a
dimension counting, we have
$${\rm dim}(M(s+1)/u\cdot M(s))=\left\{\begin{array}{ll}
1,& \text{ if }\ 2\leqslant s\leqslant l,\\
0,& \text{ otherwise}.\\
\end{array}\right.$$
Hence $u\cdot M=(\bigoplus\limits_{s=2}^l
u\cdot M(s))\bigoplus(\bigoplus\limits_{s=l+2}^{l+m}M(s))$.
Thus, as modules over $U/J(U)\cong B$, we have
$$\begin{array}{rcl}
M/u\cdot M&\cong& M(2)\bigoplus(\bigoplus\limits_{s=2}^l
M(s+1)/u\cdot M(s))\\
&\cong&\bigoplus\limits_{s=2}^{l+1}
M(1, s-2)=\bigoplus\limits_{i=0}^{l-1}
M(1, i)\cong\bigoplus\limits_{i=0}^{l-1}S_i .\\
\end{array}$$
\end{proof}

If $m=p$, then $M$ is projective since $M(p,0)$ is.
Hence by Corollaries \ref{xxcor1.6} and \ref{xxcor2.7},
we have
$$\begin{array}{rcl}
M&\cong& P(M)\cong P(M/u\cdot M)\cong
\bigoplus\limits_{i=0}^{l-1}P(S_i)\cong
\bigoplus\limits_{i=0}^{l-1}M(p, i).\\
\end{array}$$
Thus, we have the following proposition.

\begin{proposition}
\label{xxpro2.8}
Let $2\leqslant l\leqslant p$ and $r, r'\in{\mathbb Z}_p$.
Then we have the $U$-module isomorphism
$$M(l,r)\ot M(p, r')\cong
\bigoplus\limits_{i=0}^{l-1}M(p, r+r'+i).$$
\end{proposition}

\begin{proof}
We have already proven that
$M(l,0)\ot M(p, 0)\cong \bigoplus\limits_{i=0}^{l-1}M(p, i)$.
Then the proposition follows from the isomorphism and
Lemma \ref{xxlem2.2}.
\end{proof}

\begin{lemma}
\label{xxlem2.9}
Let $2\leqslant l\<m\leqslant p$ and retain the above notation.
\begin{enumerate}
\item[(1)]
$M$ contains a submodule isomorphic
to $M(l-1,1)\ot M(m-1,0)$.
\item[(2)]
For each $r\in\mathbb{Z}_p$,
$M(l,r)\ot M(m,0)$ contains a submodule
isomorphic to $M(l-1,r+1)\ot M(m-1,0)$.
\item[(3)]
For $r\in\mathbb{Z}_p$ and $1\leqslant s\leqslant l-1$,
$M(l,r)\ot M(m,0)$ contains a submodule
isomorphic to $M(l-s,r+s)\ot M(m-s,0)$.
\end{enumerate}
\end{lemma}

\begin{proof}
(1)
Recall that $\{x_i\ot y_j|1\leqslant i\leqslant l, 1\leqslant j\leqslant m\}$ 
is a basis of $M$. Let $N=M(l-1, 1)\ot M(m-1, 0)$.
Let $\{a_1, a_2, \cdots, a_{l-1}\}$ and $\{b_1, b_2, \cdots, b_{m-1}\}$
be the standard bases of $M(l-1,1)$ and $M(m-1,0)$, respectively.
Then $\{a_i\ot b_j|1\leqslant i\leqslant l-1, 1\leqslant j\leqslant m-1\}$
is a basis of $N$. By definition, we have
$w\cdot(x_i\ot y_j)=(i+j-2)x_i\ot y_j$
for all $1\leqslant i\leqslant l$ and $1\leqslant j\leqslant m$, and
$w\cdot(a_i\ot b_j)=(i+j-1)a_i\ot b_j$ for all $(i,j)$.
Now define a $\Bbbk$-linear map $f: N\ra M$ by
$$f(a_i\ot b_j)=(l-i)x_i\ot y_{j+1}+(j-m)x_{i+1}\ot y_j,$$
where $1\leqslant i\leqslant l-1$ and $1\leqslant j\leqslant m-1$.
It is easy to see that $f$ is a $\Bbbk$-linear injection and
that $f(w\cdot(a_i\ot b_j))=w\cdot f(a_i\ot b_j)$ for all
$(i,j)$.
Then by a straightforward computation, one can check that
$f(u\cdot(a_i\ot b_j))=u\cdot f(a_i\ot b_j)$ for all $(i,j)$.
This finishes the proof of part (1).

(2) This follows from part (1) and Lemma \ref{xxlem2.2}.

(3) The assertion follows from induction on $s$ and part (2).
\end{proof}

\begin{theorem}
\label{xxthm2.10}
Let $1\leqslant l\<m< p$ and suppose that $l+m\leqslant p$.
\begin{enumerate}
\item[(1)]
$$M(:=M(l,0)\ot M(m, 0))\cong\bigoplus\limits_{i=1}^l
M(m-l-1+2i,l-i).$$
\item[(2)]
Let $r, r'\in{\mathbb Z}_p$. Then
$$M(l,r)\ot M(m,r')\cong\bigoplus\limits_{i=1}^l
M(m-l-1+2i, r+r'+l-i).$$
\end{enumerate}
\end{theorem}

\begin{proof}
(1) We use induction on $l$. If $l=1$, it follows from
Lemma \ref{xxlem2.2} that the assertion holds. Now
assume $l>1$. Then $m\geqslant 2$. By Lemma \ref{xxlem2.9}(1),
there exists a submodule $N$ of $M$ such that
$N\cong M(l-1,1)\ot M(m-1, 0)$. By the induction hypothesis 
and Lemma \ref{xxlem2.2}, we have
$$\begin{array}{rcl}
N\cong S_{1}\ot M(l-1,0)\ot M(m-1, 0)
\cong\bigoplus\limits_{i=1}^{l-1}
M(m-l-1+2i, l-i).\\
\end{array}$$
Therefore, one knows that
$${\rm soc}(N)\cong \bigoplus\limits_{i=1}^{l-1}
M(1,l- i+(m-l-1+2i)-1)\cong \bigoplus\limits_{i=1}^{l-1}
S_{m+i-2}.$$

Now we use the $\Bbbk$-basis of $M$ as stated before,
and consider the submodule $\langle x_1\ot y_1\rangle$
of $M$ generated by $x_1\ot y_1$. For the convenience,
we set $x_i=0$ for $i>l$, and $y_j=0$ for $j>m$. Then
since $1<l+m\leqslant p$, we have
$$\begin{array}{rcl}
u^{l+m-2}\cdot(x_1\ot y_1)&=&
\sum\limits_{i=0}^{l+m-2}\binom{l+m-2}{i}u^i
  \cdot x_1\ot u^{l+m-2-i}\cdot y_1\\
&=&\sum\limits_{i=0}^{l+m-2}\binom{l+m-2}{i}x_{i+1}
  \ot y_{l+m-1-i}\\
&=&\binom{l+m-2}{l-1}x_{l}\ot y_m\neq 0.\\
\end{array}$$
However, $u^{l+m-1}\cdot(x_1\ot y_1)=0$. It is easy to
see that $w\cdot(x_1\ot y_1)=0$. It follows that
$\langle x_1\ot y_1\rangle$ is isomorphic to $M(l+m-1, 0)$
with ${\rm soc}(\langle x_1\ot y_1\rangle)=k(x_l\ot y_m)
\cong S_{l+m-2}$. Note that $S_{l+m-2}$ is not isomorphic
to any submodule of ${\rm soc}(N)$. Hence
$N\cap\langle x_1\ot y_1\rangle=0$,
and consequently, the sum $N+\langle x_1\ot y_1\rangle$ of
the two submodules of $M$ is a direct sum in $M$.
Thus, we have
$$\begin{aligned}
&{\rm dim}(N+\langle x_1\ot y_1\rangle)
={\rm dim}(N)+{\rm dim}(\langle x_1\ot y_1\rangle)\\
=&(l-1)(m-1)+l+m-1 =lm={\rm dim}(M).
\end{aligned}
$$
It follows that
$M=N\oplus\langle x_1\ot y_1\rangle\cong
\bigoplus\limits_{i=1}^l M(m-l-1+2i, l-i)$.

(2) It follows from Lemma \ref{xxlem2.2} and
part (1).
\end{proof}

\begin{theorem}
\label{xxthm2.11}
Let $1\leqslant l\<m<p$ and suppose that $l+m>p$.
\begin{enumerate}
\item[(1)]
$$ M\cong(\bigoplus\limits_{i=1}^{p-m}
M(m-l-1+2i, l-i))
\bigoplus(\bigoplus\limits_{i=1}^{l+m-p}
M(p, i-1)).$$
\item[(2)]
Let $r, r'\in{\mathbb Z}_p$. Then
$$M(l,r)\ot M(m,r')\cong(\bigoplus\limits_{i=1}^{p-m}
M(m-l-1+2i, r+r'+l-i))
\bigoplus(\bigoplus\limits_{i=1}^{l+m-p}
M(p, r+r'+i-1)).$$
\end{enumerate}
\end{theorem}

\begin{proof}
(1)
Since $l+m>p$ and $p> m\>l\geqslant 2$, we have
$1\leqslant l+m-p=l-(p-m)\leqslant l-1$.
It follows from Lemma \ref{xxlem2.9}(3) that
there exists a submodule $N$ of $M$ such that
$$\begin{aligned}
N&\cong M(l-(l+m-p), l+m-p))\ot M(m-(l+m-p), 0)\\
&=M(p-m, l+m)\ot M(p-l,0).
\end{aligned}
$$
Note that $(p-l)+(p-m)=2p-(l+m)<p$ and $1\<p-m\<p-l<p$.
Then by Theorem \ref{xxthm2.10}(2), one gets that
$$\begin{array}{c}
N\cong M(p-m, l+m)\ot M(p-l,0)
\cong\bigoplus\limits_{i=1}^{p-m}
M(m-l-1+2i, l-i)\\
\end{array}$$
and
$${\rm soc}(N)\cong \bigoplus\limits_{i=1}^{p-m}
M(1, l-i+(m-l-1+2i)-1)\cong \bigoplus\limits_{i=1}^{p-m}
S_{m+i-2}.$$

Now we use the $\Bbbk$-basis of $M$ as stated before.
Let $1\leqslant i\leqslant l+m-p$. Consider the submodule
$\langle x_i\ot y_1\rangle$ of $M$ generated by $x_i\ot y_1$.
At first, we have $w\cdot(x_i\ot y_1)=(i-1)x_i\ot y_1$.
For the convenience, we set $x_j=0$ for $j>l$, and
$y_j=0$ for $j>m$. Since $i+1+(p-1)=i+p\leqslant l+m$
and $p-m\leqslant l-i$, we have
$$\begin{array}{rcl}
u^{p-1}\cdot(x_i\ot y_1)
&=&\sum\limits_{j=0}^{p-1}\binom{p-1}{j}u^j\cdot x_i
   \ot u^{p-1-j}\cdot y_1\\
&=&\sum\limits_{j=0}^{p-1}\binom{p-1}{j}x_{i+j}\ot y_{p-j}\\
&=&\sum\limits_{j=p-m}^{l-i}\binom{p-1}{j}x_{i+j}
   \ot y_{p-j}\neq 0.\\
\end{array}$$
Hence $\langle x_i\ot y_1\rangle
={\rm span}\{u^j\cdot(x_i\ot y_1)|j=0, 1, \cdots p-1\}
\cong M(p, i-1)$, a projective (injective) module.
Thus, ${\rm soc}(\langle x_i\ot y_1\rangle)\cong S_{i-2}$.
Obviously,
$$S_{1-2}, S_{2-2}, \cdots, S_{(l+m-p)-2}$$
are non-isomorphic simple $U$-modules,
and none of them is isomorphic to a submodule of
$N$. It follows that the sum
$N+\sum\limits_{i=1}^{l+m-p}\langle x_i\ot y_1\rangle$
of the submodules of $M$ is direct in $M$.
Hence
$$\begin{aligned}
&{\rm dim}(N+\sum\limits_{i=1}^{l+m-p}\langle x_i\ot y_1\rangle)
={\rm dim}(N)+\sum\limits_{i=1}^{l+m-p}
{\rm dim}(\langle x_i\ot y_1\rangle)\\
=&(p-m)(p-l)+p(l+m-p)=lm={\rm dim}(M).
\end{aligned}
$$
Thus, we have
$$\begin{array}{rcl}
M&=&N\bigoplus(\bigoplus\limits_{i=1}^{l+m-p}\langle x_i\ot y_1\rangle)\\
&\cong&(\bigoplus\limits_{i=1}^{p-m}
M(m-l-1+2i, l-i))
\bigoplus(\bigoplus\limits_{i=1}^{l+m-p}
M(p, i-1)).\\
\end{array}$$

(2) It follows from Lemma \ref{xxlem2.2} and
part (1).
\end{proof}

\begin{corollary}
\label{xxcor2.12}
Let $1\leqslant l\<m<p$ and $r, r'\in{\mathbb Z}_p$. 
\begin{enumerate}
\item[(1)] 
There is a simple summand in $M(l,r)\ot M(m,r')$ if and 
only if $l=m$.
\item[(2)] 
If $l=m=2<p$, then $M(2, r)\ot M(2, r')\cong
S_{r+r'+1}\oplus M(3, r+r')$.
\end{enumerate}
\end{corollary}

\begin{proof}
It follows from Corollary \ref{xxcor1.6}, Theorems \ref{xxthm2.10} 
and \ref{xxthm2.11}.
\end{proof}

\begin{proof}[Proof of Theorem \ref{xxthm0.2}]
(1) This is Theorem \ref{xxthm2.10}.

(2) If $m=p$, this follows from Lemma \ref{xxlem2.2} and
Proposition \ref{xxpro2.8}. If $m<p$, it is Theorem \ref{xxthm2.11}.
\end{proof}

Throughout the rest of this section, let $a=[S_1]$ and
$x=[M(2,0)]$ in the Green ring $r(U)$ of $U$. Since $U$
is cocommutative, $r(U)$ is a commutative ring. The
following lemma is similar to \cite[Lemma 3.8]{CVZ}
(see \cite[Lemma 3.8]{CVZ} for a proof.)

\begin{lemma}
\label{xxlem2.13}
Retain the above notation. The following hold.
\begin{enumerate}
\item[(1)]
$a^p=1$ and $[M(l, r)]=a^r[M(l,0)]$ for all
$2\leqslant l\leqslant p$ and $r\in{\mathbb Z}_p$.
\item[(2)]
If $p>2$, then $[M(l+1, 0)]=x[M(l,0)]-a[M(l-1,0)]$
for all $2\leqslant l\leqslant p-1$.
\item[(3)]
$x[M(p,0)]=(a+1)[M(p,0)]$.
\item[(4)]
$r(U)$ is generated by $a$ and $x$ as a ring.
\end{enumerate}
\end{lemma}


Note that Lemma \ref{xxlem2.13}(1) is slightly different
from \cite[Lemma 3.8(1)]{CVZ}. The following is
similar to \cite[Corollary 3.9]{CVZ} and its proof is omitted.

\begin{corollary}
\label{xxcor2.14}
Let $u_1, u_2, \cdots, u_p$ be a series of elements
of the ring $r(U)$ defined recursively by $u_1=1$,
$u_2=x$ and
$$u_l=xu_{l-1}-au_{l-2},\  p\geqslant l\geqslant 3.$$
Then $[M(l, 0)]=u_l$ for all $1\leqslant l\leqslant p$
and $(x-a-1)u_p=0$.
\end{corollary}


Let $R$ be the subring of $r(U)$ generated by $a$,
and $\langle a\rangle$ the subgroup of the group
of the invertible elements of $r(U)$ generated by
$a$. Then $\langle a\rangle$ is a cyclic group of
order $p$ by Lemma \ref{xxlem2.13}(1), and
$R={\mathbb Z}\langle a\rangle$ is the group ring
of $\langle a\rangle$ over $\mathbb Z$. Let
${\mathbb Z}[y,z]$ be the polynomial algebra over
$\mathbb Z$ in two variables $y$ and $z$. Define
$f_n(y,z)\in{\mathbb Z}[y,z]$, $n\geqslant 1$,
recursively, by $f_1(y,z)=1$, $f_2(y,z)=z$ and
$$f_n(y,z)=zf_{n-1}(y,z)-yf_{n-2}(y,z),\ n\geqslant 3.$$
Then by \cite[Lemma 3.11]{CVZ}, for any $n\geqslant 1$,
we have
$$f_{n}(y, z)=\sum\limits_{i=0}^{[(n-1)/2]}
(-1)^i\binom{n-1-i}{i}y^iz^{n-1-2i},$$
where $[(n-1)/2]$ denotes the integer part of
$(n-1)/2$. Hence deg$_z(f_n(y,z))=n-1$ for all
$n\geqslant 1$, where ${\rm deg}_z(f(y,z))$ denotes
the degree of $f(y,z)\in{\mathbb Z}[y,z]$ in $z$.
See \cite[Section 3]{CVZ} for more information about
$f_n(y,z)$. Let $I=(y^p-1, (z-y-1)f_p(y,z))$ be
the ideal of ${\mathbb Z}[y,z]$ generated by $y^p-1$
and $(z-y-1)f_p(y,z)$.

With the above notations, we have the following
corollary that is similar to \cite[Theorem 3.10]{CVZ}.
See the proof of \cite[Theorem 3.10]{CVZ} for some
details.

\begin{corollary}
\label{xxcor2.15}
The Green ring $r(U)$ is isomorphic to the
factor ring ${\mathbb Z}[y,z]/I$.
\end{corollary}


\begin{corollary}
\label{xxcor2.16}
The Green ring $r(U)$ is isomorphic to
the Green ring $r(H_p(q))$ where $H_p(q)$ is the Taft algebra
of rank $p$ {\rm{(}}over a possibly different base field
{\rm{)}}.
\end{corollary}

\begin{proof} This is clear by comparing Corollary
\ref{xxcor2.15} with \cite[Theorem 3.10]{CVZ}.
\end{proof}

\section{$U$-actions on AS regular algebras}
\label{xxsec3}

Recall from \cite[p. 171]{AS} that a connected graded algebra
$T$ is called {\it Artin-Schelter regular} (or {\it AS regular},
for short) of dimension $d$ if the following hold:
\begin{enumerate}
\item[(a)]
$T$ has global dimension $d<\infty$,
\item[(b)]
$\Ext^i_T({}_T\Bbbk, {}_{T}T)=\Ext^i_{T}(\Bbbk_T,T_T)=0$ for all
$i\neq d$, where $\Bbbk=T/T_{\geq 1}$,
\item[(c)]
$\Ext^d_T({}_T\Bbbk, {}_{T}T)\cong \Ext^d_{T}(\Bbbk_T,T_T)\cong
\Bbbk(l)$ for some integer $l$,
\item[(d)]
$T$ has finite Gelfand--Kirillov dimension, see
\cite[Definition 1.7]{KWZ}.
\end{enumerate}

We will use the following general setting.

\begin{enumerate}
\item[$\bullet$]
Let $T$ be a noetherian connected graded AS regular algebra.
\item[$\bullet$]
Let $H$ be a finite-dimensional Hopf algebra acting on $T$
inner-faithfully and homogeneously (namely, each degree $i$
piece $T_i$ of $T$ is a left $H$-submodule of $T$), such that
$T$ is a left $H$-module algebra.
\end{enumerate}

For any $H$-action on $T$, the fixed subring of the action
is defined to be
$$T^H:=\{ a\in T\mid h\cdot a=\epsilon(h) a,\; \forall\; h\in H\}.$$

\begin{lemma}
\label{xxlem3.1}
Let $T$ be a graded algebra generated in degree 1 and let $U$
act on $T$ inner-faithfully. Then $V:=T_1$ is a direct sum of
indecomposable left $U$-modules, and at least one of which is
not 1-dimensional. As a consequence, $\dim_{\Bbbk} V\geq 2$.
\end{lemma}

\begin{proof}
Note that $U$ has three Hopf ideals, namely, $0$,
$\ker \epsilon$, and the ideal generated by $u$. If $V$
is a direct sum of 1-dimensional left $U$-modules, then $u \cdot V=0$.
Then $u \cdot T=0$ since $T$ is generated by $V$. So the
$U$-action is not inner-faithful, yielding a contradiction. The
assertion follows.
\end{proof}

Note that commutative AS regular algebras are exactly
commutative polynomial rings. Let $T$ be a commutative polynomial
ring $\Bbbk[x_1,\cdots,x_n]$ with $\deg x_i=1$ for all $i$. Since
$U$ is cocommutative, every $U$-action on $T$ is uniquely
induced and uniquely determined by its action on the degree 1 piece.
The following classifies completely all $U$-actions on noetherian AS
regular algebras of global dimension 2.

\begin{proposition}
\label{xxpro3.2}
Let $U$ act inner-faithfully on a noetherian Koszul AS regular
algebra $T$ of global dimension 2.
\begin{enumerate}
\item[(1)]
$T$ is commutative, namely, $T=\Bbbk[V]$ where $V=T_1$.
As a consequence, $V$ is a 2-dimensional indecomposable
left $U$-module.
\item[(2)]
Using the notation introduced in Convention \ref{xxcon1.4},
we write $V$ as $M(2, i)=\Bbbk x_1 \oplus \Bbbk x_2$ for
some $0\< i\< p-1$. Then the following hold.
\begin{enumerate}
\item[(2a)]
If $V=M(2, p-1)$, then $T^U=\Bbbk[x_1^p, x_2]$.
\item[(2b)]
If $V=M(2,i)$ for some $0\leq i\leq p-2$, then
$T^U=\Bbbk[x_1^p,x_2^p]$.
\end{enumerate}
\end{enumerate}
\end{proposition}

\begin{proof}
(1) Since the $U$-action on $T$ is inner-faithful, $V$ is
not a direct sum of two 1-dimensional simples by Lemma
\ref{xxlem3.1}. Hence $V=M(2,i)$ for some $i$, and then
$T=\Bbbk\langle V\rangle /(r)$ where $r\in V\otimes V$ is
the relation of $T$. By Lemma \ref{xxlem2.2} and Corollary
\ref{xxcor2.6}, the socle of $V\otimes V$ is two dimensional,
spanned by $z_3:=-x_1\otimes x_2+x_2\otimes x_1$ and
$z_4:=x_2\otimes x_2$. Since $\Bbbk r$ is a left $U$-module,
it must be either $\Bbbk z_3$ or $\Bbbk z_4$. Since an AS
regular algebra of global dimension two is a domain, $r$
cannot be $z_4$. Therefore $r=z_3$, and $T$ is commutative.

(2) By part (1), $T=\Bbbk[x_1,x_2]$. Since $u$ and $w$ are
primitive, both of them act on $T$ as derivatives. Then it is
easy to show that $x_1^p, x_2^p\in T^U$.

Let $V=M(2,i)$. For any $0\leq a,b\leq p-1$, it is
straightforward to check that
\begin{equation}
\label{E3.2.1}\tag{E3.2.1}
u\cdot (x_1^a x_2^b)= a x_1^{a-1} x_2^{b+1},
\quad {\text{and}} \quad
w\cdot (x_1^a x_2^b)=(a i+ b(i+1)) x_1^a x_2^b.
\end{equation}

\begin{enumerate}
\item[(2a)]
If $i=p-1$, then \eqref{E3.2.1} implies that
$x_2\in T^U$ and that $T^U=\Bbbk[x_1^p, x_2]$.
\item[(2b)]
If $i\neq p-1$, then one sees that
$T^U=\Bbbk[x_1^p, x_2^p]$ by \eqref{E3.2.1}.
\end{enumerate}
\end{proof}

For a generalization of Proposition \ref{xxpro3.2}(1),
see Proposition \ref{xxpro0.4} (and Proposition \ref{xxpro4.3}).

Next we would like to determine $U$-actions on noetherian Koszul
AS regular algebras of global dimension $3$. We will use the fact
that a noetherian connected graded algebra of global dimension three
is a domain \cite[Theorem]{Sn}.

\begin{convention}
\label{xxcon3.3}
The following are assumed for the rest of this section.
\begin{enumerate}
\item[(1)]
Let $U$ act inner-faithfully and homogeneously on a
noetherian Koszul AS regular algebra $T$ of global dimension 3.
\item[(2)]
Let $V:=T_1$ be the degree 1 piece of the algebra $T$. It is
well-known that $V$ is a 3-dimensional left $U$-module.
\item[(3)]
Let $R\subseteq V\otimes V$ be the relation space of $T$,
namely, $T=\Bbbk\langle V\rangle/(R)$. It is well-known that
$R$ is a 3-dimensional left $U$-submodule of $V\otimes V$.
\item[(4)]
If $V$ is $M(3,i)$, then we write
$V=\Bbbk x_1\oplus \Bbbk x_2\oplus \Bbbk x_3$
using the notation introduced in Convention \ref{xxcon1.4}.
\item[(5)]
If $V$ is $M(2,i)\oplus S_j$ (where $S_j=M(1,j)$), then we write
$M(2,i)=\Bbbk x_1\oplus \Bbbk x_2$ using the notation introduced
in Convention \ref{xxcon1.4} and $S_j=\Bbbk y$.
\item[(6)]
Since $T$ is noetherian Koszul AS regular of global dimension three,
the Hilbert series of $T$ is $(1-t)^{-3}$. In particular,
$\dim_{\Bbbk} T_3=10$.
\end{enumerate}
\end{convention}

There is a large class of noetherian Koszul AS regular algebras
of global dimension 3 and the classification of such algebras
over a field of positive characteristic has not been done. We can
classify all $U$-actions on noetherian Koszul AS regular algebras
of global dimension 3 because there is no inner-faithful $U$-action
on a generic AS regular algebra, see Proposition \ref{xxpro4.3}.
The basic idea in the following is to work out the left
$U$-module $R$ which is a $U$-submodule of $V\otimes V$.

\begin{lemma}
\label{xxlem3.4}
Suppose $V$ is $M(3,i)$. Let
$z_4:=-x_1\otimes x_3+x_2\otimes x_2-x_3\otimes x_1$,
$z_5:=x_2\otimes x_3-x_3\otimes x_2$ and $z_6:=-x_3\otimes x_3$.
The following hold.
\begin{enumerate}
\item[(1)]
$soc(V\otimes V)= S_{2i+2}\oplus S_{2i+3}\oplus S_{2i+4}$
where basis elements for simple modules $S_{2i+2}$, $S_{2i+3}$,
$S_{2i+4}$ respectively are $z_4$, $z_5$, $z_6$ respectively.
\item[(2)]
If $p=3$, then $V\otimes V=M(3, 2i)\oplus M(3,2i+1)
\oplus M(3,2i+2)$ where the socles of $M(3,2i)$, $M(3,2i+1)$,
$M(3,2i+2)$ are generated by $z_4$, $z_5$, $z_6$ respectively.
\item[(3)]
If $p\geq 5$, then $V\otimes V=M(1,2+2i)\oplus M(3,1+2i)\oplus M(5, 2i)$
where the socles of $M(1,2+2i)$, $M(3,1+2i)$, $M(5, 2i)$ respectively
are generated by $z_4$, $z_5$, $z_6$ respectively.
\end{enumerate}
\end{lemma}

\begin{proof}
(1) This follows from Lemma \ref{xxlem2.2} and
Corollary \ref{xxcor2.6}.

(2) This follows from part (1) and Proposition \ref{xxpro2.8}.

(3) This follows from part (1) and Theorem \ref{xxthm2.10}(2)
when $p>5$ and Theorem \ref{xxthm2.11}(2) when $p=5$.
\end{proof}

\begin{lemma}
\label{xxlem3.5}
Suppose $V=M(3,i)=\Bbbk x_1\oplus \Bbbk x_2\oplus \Bbbk x_3$.
Let $W=x_3\otimes V+V\otimes x_3$.
\begin{enumerate}
\item[(1)]
$R\cap (x_3 \otimes V)=R\cap (V\otimes x_3)=0$ and
$\dim_{\Bbbk} (R\cap W)\leq 2$.
\item[(2)]
$\dim_{\Bbbk} (R\cap W)\neq 1$.
\item[(3)]
$R$ does not contain three linearly
independent elements of the form
$$\begin{aligned}
f_1 &:= x_1\otimes x_2 + \xi_1,\\
f_2&:= x_2\otimes x_1 +\xi_2,\\
f_3&:= x_2\otimes x_2 +\xi_3,
\end{aligned}
$$
where $\xi_1,\xi_2,\xi_3\in W$.
\item[(4)]
$\dim_{\Bbbk} (R\cap W)\neq 0$.
\item[(5)]
$\dim_{\Bbbk} (R\cap W)=2$, $x_3$ is a normal element in $T$, and
one of the following occurs:
\begin{enumerate}
\item[(5a)]
$p\geq 3$, $T$ is commutative.
\item[(5b)]
$p=3$ and the relations of $T$ are
$$x_2x_1-x_1x_2+x_3^2=0, \quad
x_3x_1-x_1x_3=0, \quad
x_3x_2-x_2x_3=0.$$
\end{enumerate}
\end{enumerate}
\end{lemma}

\begin{proof} Some non-essential computations are skipped.
Since $V=M(3,i)$, $p\geq 3$.

(1) If $R\cap (x_3\otimes V)\neq 0$, then there is a relation
of the form $x_3 \otimes v=0$ for some
$0\neq v\in V$. This contradicts the fact that $T$ is a domain.
Therefore $R\cap (x_3\otimes V)= 0$. By symmetry,
$R\cap (V\otimes x_3)= 0$.

The inclusion $R\cap W\to W$ induces an injective map
$R\cap W\to W/(V\otimes x_3)$.
Therefore $\dim_{\Bbbk} R\cap W\leq \dim_{\Bbbk} W/(V\otimes x_3)=2$.

(2) Suppose to the contrary that $\dim_{\Bbbk} R\cap W=1$.
Let $f$ be a basis element in $R\cap W$ and write it as
$$f=a x_3\otimes x_1 + b x_3\otimes x_2 + v\otimes x_3$$
for some $v\in V$. By part (1), either $a$ or $b$ is nonzero.
Suppose first that $a\neq 0$. Then $u\cdot f= a x_3\otimes x_2
+w\otimes x_3$ for some $w\in V$. So $\{f,u\cdot f\}$
are linearly independent elements in $R\cap W$, yielding a
contradiction. Therefore $a=0$, and in this case we may assume
that $b=1$ and
$$f=x_3\otimes x_2+ a_1 x_1\otimes x_3- q x_2\otimes x_3 +c_1
x_3\otimes x_3.$$
Note that
$$(2i+3-w)\cdot f= a_1 x_1\otimes x_3- c_1 x_3\otimes x_3$$
which must be zero as $\dim_{\Bbbk} R\cap W=1$. Therefore $a_1=c_1=0$
and $f=x_3\otimes x_2-q x_2 \otimes x_3$. Since $\Bbbk f$ is
in the socle of $V\otimes V$, by Lemma \ref{xxlem3.4},
we have that $f_1:=z_5=x_2\otimes x_3- x_3 \otimes x_2\in R\cap W$.

Now let $g\in R\setminus W$ and write it as
$$g=a_{11} x_1\otimes x_1 +a_{12} x_1\otimes x_2+ a_{21}
x_2\otimes x_1+a_{22}x_2\otimes x_2 +\phi_0$$
where $\phi_0\in W$. Since $\dim_{\Bbbk} R/R\cap W=
\dim_{\Bbbk} R-\dim_{\Bbbk} (R\cap W)=2$, $u^2\cdot g\in W$.
By a computation, $u^2 \cdot g= 2 a_{11} x_2\otimes x_2+\phi_1$
where $\phi_1\in W$. Thus $a_{11}=0$. In this case
$$u\cdot g=(a_{12}+a_{21}) x_2\otimes x_2 +a_{12} x_1\otimes x_3+ 
a_{21} x_3\otimes x_1+\phi_2$$
where $\phi_2\in W$.
If $3(a_{12}+a_{21})\neq 0$, then 
$$\begin{aligned}
u^2\cdot g&\equiv (2a_{12}+a_{21})x_2\otimes x_3
+(a_{12}+2a_{21})x_3\otimes x_2 \mod \Bbbk x_3\otimes x_3\\
&\equiv 3(a_{12}+a_{21})x_2\otimes x_3 \mod \Bbbk x_3\otimes x_3
+\Bbbk z_5
\end{aligned}$$
as $u^2\cdot (a_{22} x_2\otimes x_2+\phi_0)\equiv 0$
modulo $\Bbbk x_3\otimes x_3$. Then $u^2\cdot g$ and
$z_5$ are linearly independent elements in $R\cap W$,
yielding a contradiction. Therefore $3(a_{12}+a_{21})=0$.

Since $R/R\cap W$ has dimension two, either $a_{12}\neq 0$
or $a_{21}\neq 0$. By symmetry, we assume that $a_{12}\neq 0$.
So we can assume that $a_{12}=-1$.
We need to consider the following two cases:

Case (2a): $a_{21}=-a_{12}=1$.
Then $R$ has two elements of the form
$$\begin{aligned}
f_2&=x_2\otimes x_1-x_1\otimes x_2
+a_1 x_1\otimes x_3+b_1 x_3\otimes x_1+ c_1 x_2\otimes x_3+d_1 x_3\otimes x_3,\\
f_3&=x_2\otimes x_2
+a_2 x_1\otimes x_3+b_2 x_3\otimes x_1+ c_2 x_2\otimes x_3+d_2 x_3\otimes x_3.
\end{aligned}
$$
It is easy to see that
$$u\cdot f_2=x_3\otimes x_1-x_1\otimes x_3
+a_1 x_2\otimes x_3+b_1 x_3\otimes x_2+ c_1 x_3\otimes x_3$$
which is in $R\cap W$ but not in $\Bbbk f_1$. This yields a
contradiction.

Case (2b): $q:=a_{21}\neq -a_{12}=1$. Since $3(a_{12}+a_{21})=0$, we obtain
that $p=3$. We also have relations similar to $f_2$ and $f_3$ in Case (2a).
Since $R$ is a left $U$-module, we can choose $f_2,f_3$ so that 
$(2i+1-w)\cdot f_2=0=(2i+2-w)\cdot f_3$. In this case, we have
$$\begin{aligned}
f_2&=qx_2\otimes x_1-x_1\otimes x_2 +d x_3\otimes x_3,\\
f_3&=x_2\otimes x_2
+a x_1\otimes x_3+b x_3\otimes x_1.
\end{aligned}
$$
By easy calculation, 
$$\begin{aligned}
u\cdot f_2&=(q-1) x_2\otimes x_2 -x_1\otimes x_3+q x_3\otimes x_1,\\
u\cdot f_3&=(1+a)x_2\otimes x_3+ (1+b) x_3\otimes x_2.
\end{aligned}
$$
Since $u\cdot f_2$ and $u\cdot f_3$ are relations of $T$, by comparing
$u\cdot f_2$ with $f_3$ and $u\cdot f_3$ with $f_1$ and using the 
fact that $p=3$, we obtain 
that $a=-(q-1)^{-1}$ and $b=q(q-1)^{-1}$. This algebra can be built
from $\Bbbk[x_2,x_3]$ by adding $x_1$ with relations $f_2$ and $f_3$.
So if $T$ is AS regular (consequently having Hilbert series $(1-t)^{-3}$),
it must be an Ore extension of the form $\Bbbk[x_2,x_3][x_1;\sigma,\delta]$
for some automorphism $\sigma$ and $\sigma$-derivation $\delta$
of the polynomial ring $\Bbbk[x_2,x_3]$. Now an easy ring 
theory argument shows that the existence of $(\sigma,\delta)$
forces that $q=1$, yielding a contradiction. 

(3) Suppose to the contrary that $R$ contains three linearly
independent elements of the form
$$\begin{aligned}
f_1 &= x_1\otimes x_2 + \xi_1,\\
f_2&= x_2\otimes x_1 +\xi_2,\\
f_3&= x_2\otimes x_2 +\xi_3,
\end{aligned}
$$
where $\xi_1,\xi_2,\xi_3\in W$.
Then $u\cdot f_1$ and $u\cdot f_2$ must be equal to
$f_3$. Then $R/(\Bbbk f_3)$ is a direct sum of
two 1-dimensional simples. Then $R$ is not indecomposable.
Then $R$ contains at least two of elements $z_4,z_5,z_6$
by Lemma \ref{xxlem3.4}(1). Clearly $z_5$ and $z_6$ are
not a linear combination of these $f_i$. This
yields a contradiction.

(4) On the contrary we suppose that $R\cap W=0$.

First we assume that $p>3$.
If $f=x_1\otimes x_1+\sum_{(i,j)\neq (1,1)} a_{i,j} x_i
\otimes x_j$ is in $R$, then $u^3\cdot f=3(x_2\otimes x_3+x_3\otimes x_2)
+c x_3\otimes x_3$ for some $c\in\Bbbk$,
which is a nonzero element in $R\cap W$, a contradiction.
So $R$ does not contain any element of the form
$x_1\otimes x_1+\sum_{(i,j)\neq (1,1)} a_{i,j} x_i \otimes x_j$.
Since $R\cap W=0$, $R$ must contain three elements of the
form given in part (3). This contradicts part (3).

Now we assume that $p=3$.
By Lemma \ref{xxlem3.4}(1), one of the relations is $z_4\in R$
as $z_5$ and $z_6$ are in $W$.
By part (3), we have a relation $f\in R$ of the form
$$f=x_1\otimes x_1 + {\text{higher terms}}.$$
Then $u\cdot f= x_1\otimes x_2+x_2\otimes x_1 + {\text{higher terms}}$.
Modulo $z_4$ and $u\cdot f$, we can assume
that
$$f= x_1\otimes x_1+ a x_1\otimes x_2 + b x_1\otimes x_3
+c x_3\otimes x_1 + d x_2\otimes x_3 + e x_3\otimes x_2 +g x_3\otimes x_3.$$
Then
$$(2i-w)\cdot f
=-a x_1\otimes x_2 -2 b x_1\otimes x_3- 2c x_3\otimes x_1 -g x_3\otimes x_3.$$
Since $(2i-w)\cdot f\in R$ and it is not a nonzero linear
combination of $z_4$, $f$ and $u\cdot f$, it must be zero. Therefore
$a=b=c=g=0$ and
$$f= x_1\otimes x_1+ d x_2\otimes x_3 + e x_3\otimes x_2.$$
Consequently, we have two more relations
$$\begin{aligned}
u\cdot f&=x_1\otimes x_2+x_2\otimes x_1 +(d+e) x_3\otimes x_3,\\
-u^2\cdot f&=-x_1\otimes x_3 +x_2\otimes x_2 -x_3\otimes x_1.
\end{aligned}
$$
Since $T$ is a domain one of $d$ and $e$ is nonzero.
By symmetry, we can assume $e\neq 0$. After changing a basis
element, we may assume that $e=1$. Next we consider two cases
dependent on whether $d+e$ is zero or not.

Case 1: $d+e=0$.  Since $e=1$, $d=-1$. So three relations of
$T$ are, after we omit the $\otimes$ symbol:
$$x_3x_2=x_2x_3-x_1^2, \quad
x_2x_1=-x_1x_2, \quad x_3x_1=-x_1x_3+x_2^2.$$
Following ideas from Bergman's Diamond lemma \cite{Be}, the next
computation is referred to as {\it resolving the overlap ambiguity} of
$x_3(x_2x_1)=(x_3x_2)x_1$. Using the order $x_1<x_2<x_3$ in the
algebra $T$, we have
$$\begin{aligned}
(x_3x_2)x_1&=(x_2x_3-x_1^2)x_1=x_2(x_3x_1)-x_1^3\\
&=x_2(-x_1x_3+x_2^2)-x_1^3=x_1x_2x_3+x_2^3-x_1^3,\\
x_3(x_2x_1)&=x_3(-x_1x_2)=-(x_3x_1)x_2\\
&=-(-x_1x_3+x_2^2)x_2=x_1(x_3x_2)-x_2^3\\
&=x_1(x_2x_3-x_1^2)-x_2^3=x_1x_2x_3-x_1^3-x_2^3.
\end{aligned}
$$
Since $x_3(x_2x_1)=(x_3x_2)x_1$, we obtain that
$x_1x_2x_3+x_2^3-x_1^3=x_1x_2x_3-x_1^3-x_2^3$, or equivalently,
$x_2^3=0$. But $T$ is a domain, so this case cannot happen.

Case 2: $d+e=d+1\neq 0$. In this case, the relations of $T$ are
$$
x_3x_2=-d x_2x_3-x_1^2, \quad
x_3^2=-(d+e)^{-1}(x_2x_1+x_1x_2), \quad x_3x_1=-x_1x_3+x_2^2.$$
By resolving the overlap ambiguity of $(x_3^2) x_3=x_3(x_3^2)$
(details omitted, same as below), we obtain that
$$(d-1) x_2^3 =(d-1) (x_2x_1+x_1x_2)x_3.$$
By resolving the overlap ambiguity of $(x_3)^2 x_1=x_3(x_3x_1)$,
we obtain that
$$(d^2-1)x_2^2x_3=(d+(d+1)^{-1})x_2x_1^2-(1+(d+1)^{-1})x_1^2x_2.$$
By resolving the overlap ambiguity of $(x_3^2)x_2=x_3(x_3x_2)$,
we obtain that
$$(1+d^2(d+1)^{-1})x_2^2x_1
=(d-1)x_1^2x_3+(1+(d+1)^{-1})^{-1}x_1x_2^2
+(d+1)^{-1}(1-d^2)x_2x_1x_2.$$
If $d\neq 1$, by using the relations of degree 2 and three relations
coming from resolving the overlap ambiguities, then $T_3$ is a $\Bbbk$-span of
$$\{x_1^3, x_1^2x_2, x_1x_2x_1, x_1x_2^2,x_2x_1^2,x_2x_1x_2,
x_1^2x_3, x_1x_2x_3,x_2x_1x_3\}.$$
As a consequence, $\dim_{\Bbbk} T_3=9<10$, yielding
a contradiction. Therefore $d=1$. We already have $e=1$,
so the three relations of $T$ become
$$\begin{aligned}
x_1x_3+x_3x_1&=x_2^2,\\
x_1x_2+x_2x_1&=-2 x_3^2 =x_3^2,\\
x_3x_2+x_2x_3&=-x_1^2.
\end{aligned}
$$
Note that, after setting $x_1\to -x_1$ in
the above algebra, the new algebra, denoted by $T'$, has relations
$$\begin{aligned}
x_1x_3+x_3x_1&=-x_2^2,\\
x_1x_2+x_2x_1&=-x_3^2,\\
x_3x_2+x_2x_3&=-x_1^2.
\end{aligned}
$$
Its Koszul dual, denoted by $B$, is a commutative algebra
generated by $y_1,y_2,y_3$ subject to relations:
$$\begin{aligned}
y_1y_3-y_3y_1=0, \quad & y_1y_3-y_2^2=0,\\
y_1y_2-y_2y_1=0, \quad & y_1y_2-y_3^2=0,\\
y_3y_2-y_2y_3=0, \quad & y_2y_3-y_1^2=0.
\end{aligned}
$$
It is easy to see that there is a surjective algebra map
from $B$ to $\Bbbk [t]$ by setting $y_i\to t$ for $i=1,2,3$.
Therefore $B$ is not finite dimensional Frobenius.
By \cite[Proposition 5.10]{Sm} or \cite[Corollary D]{LPWZ},
$T'$ is not AS regular. As a consequence, $T$ is not AS regular,
yielding a contradiction. Thus the assertion follows.

(5) By parts (1,2,4), $\dim_{\Bbbk} R\cap W=2$. By the proof of part (1),
the natural $\Bbbk$-linear map $\phi: R\cap W\to W/(V\otimes x_3)$
is an isomorphism. Let $f_1,f_2$ be two linearly independent elements
in $R\cap W$. Since $\phi(f_1)$ and $\phi(f_2)$ are linearly
independent, we may assume $\phi(f_1)=x_3\otimes x_1$ and
$\phi(f_2)=x_3\otimes x_2$. In other words, we can write
\begin{align}
\notag
f_1&:= x_3\otimes x_1+ a_1 x_1\otimes x_3+ b_1 x_2 \otimes x_3+
c_1 x_3\otimes x_3,\\
\notag
f_2&:= x_3\otimes x_2+ a_2 x_1\otimes x_3+ b_2 x_2 \otimes x_3+
c_2 x_3\otimes x_3.
\end{align}
Therefore $x_3 T=Tx_3$ and hence, $x_3$ is a normal element
in $T$.

Since $x_3$ is normal and $\Bbbk x_3$ is a left $U$-module,
the $U$-action on $T$ induces a natural $U$-action on
$Z:=T/(x_3)$ where $Z$ is a noetherian Koszul AS regular
algebra of global dimension two. Note that the degree 1 piece
of $Z$ is $Z_1=M(2,i)$, whence the $U$-action on $Z$ is
inner-faithful. By Proposition \ref{xxpro3.2}(1), $Z$ is
commutative, so $x_2 x_1-x_1 x_2=0$ in $Z$. This implies that
the third relation in $T$ is
\begin{align}
\notag
f_3 &:= x_2 \otimes x_1-x_1 \otimes x_2 +
a_3 x_1 \otimes x_3 + b_3 x_2 \otimes x_3 +c_3
x_3\otimes x_3.
\end{align}
Easy computations show that
\begin{align}
\label{E3.5.1}\tag{E3.5.1}
u\cdot f_1&=x_3\otimes x_2+a_1 x_2\otimes x_3+ b_1 x_3\otimes x_3,\\
\label{E3.5.2}\tag{E3.5.2}
u\cdot f_2&=x_3\otimes x_3+a_2 x_2 \otimes x_3 +b_2 x_3 \otimes x_3,\\
\label{E3.5.3}\tag{E3.5.3}
u\cdot f_3&=x_3\otimes x_1 -x_1\otimes x_3+a_3 x_2\otimes x_3+b_3 x_3\otimes x_3,\\
\label{E3.5.4}\tag{E3.5.4}
(2i+2-w)\cdot f_1&=-b_1 x_2\otimes x_3-2c_1 x_3\otimes x_3,\\
\label{E3.5.5}\tag{E3.5.5}
(2i+3-w)\cdot f_2&=a_2 x_1\otimes x_3 -c_2 x_3\otimes x_3,\\
\label{E3.5.6}\tag{E3.5.6}
(2i+1-w)\cdot f_3&=-a_3 x_1\otimes x_3 -2 b_3 x_2\otimes x_3 -3 c_3 x_3\otimes x_3.
\end{align}
Since $R$ is a $U$-submodule, the above elements are in
$R\cap W$. Equations \eqref{E3.5.4}-\eqref{E3.5.6} imply that
$b_1=c_1=a_2=c_2=a_3=b_3 =3c_3=0$. Equation \eqref{E3.5.3} says
that $f_1=x_3\otimes x_1-x_1\otimes x_3$ and that $a_1=-1$,
Equation \eqref{E3.5.1} says that $f_2=x_3\otimes x_2-x_2\otimes x_3$
and that $b_2=-1$. Combining all these we obtain two cases:
\begin{enumerate}
\item[(a)]
$p\geq 3$, $T$ is commutative, or
\item[(b)]
$p=3$ and relations of $T$ are, up to a change of basis,
$$x_2\otimes x_1-x_1\otimes x_2+x_3\otimes x_3=
x_3\otimes x_1-x_1\otimes x_3=x_3\otimes x_2-x_2\otimes x_3=0.$$
\end{enumerate}
This finishes the proof.
\end{proof}

We will recycle the letter $z_4$ with a different meaning in
the next lemma. Similar to Lemma \ref{xxlem3.4}, we have

\begin{lemma}
\label{xxlem3.6}
Suppose $V=M(2,i)\oplus S_j=\Bbbk x_1\oplus \Bbbk x_2
\oplus \Bbbk y$ for some $i,j \in {\mathbb Z}_p$.
\begin{enumerate}
\item[(1)]
$V\otimes V=[M(2,i)\otimes M(2,i)]\oplus [M(2,i)\otimes S_j
\oplus S_j\otimes M(2,i)]\oplus S_j\otimes S_j$.
\item[(2)]
$soc(V\otimes V)=[S_{2i+1}\oplus S_{2i+2}]
\oplus [S_{i+1+j}\oplus S_{j+i+1}]
\oplus S_{2j}$ with corresponding basis
elements $z_3=-x_1\otimes x_2+x_2\otimes x_1$
and $z_4=x_2\otimes x_2$ for
$S_{2i+1}\oplus S_{2i+2}$, $x_2\otimes y, y\otimes x_2$
for $S_{i+1+j}\oplus S_{j+i+1}$, and
$y\otimes y$ for $S_{2j}$.
\end{enumerate}
\end{lemma}

\begin{lemma}
\label{xxlem3.7}
Suppose $V=M(2,i)\oplus S_j$ as in Lemma {\rm{\ref{xxlem3.6}}}.
Then $u^2 \cdot R=0$ and $\dim_{\Bbbk} u\cdot R\leq 1$.
\end{lemma}

\begin{proof}
Suppose to the contrary that $u^2 \cdot R\neq 0$. Then the subspace
$M(2,i)\otimes M(2,i)$
in Lemma \ref{xxlem3.6}(1) has a direct summand that is a
three-dimensional indecomposable module $M(3, 2i)$ (using Corollary
\ref{xxcor2.12}) with socle $\Bbbk z_4=\Bbbk x_2\otimes x_2$. As a
consequence $p\geq 3$. In this case, Lemma \ref{xxlem3.6}(1)
implies that $V\otimes V=M(3,2i)\oplus N$ where $u^2\cdot N=0$.
Pick $f\in R$ such that $u^2\cdot f\neq 0$ and write
$f=f_0+f_1$ where $f_0\in M(3,2i)$ and $f_1\in N$.
Then $u^2 \cdot f_0 =u^2\cdot f\in R$. Since $u^2 \cdot f_0$
is in the socle of $M(3,2i)$, we obtain that $u^2\cdot f$
is a nonzero element in $\Bbbk z_4$. Thus $x_2^2=0$ in $T$,
which contradicts the fact that every noetherian Koszul AS regular
algebra of global dimension three is a domain. Therefore
$u^2\cdot R=0$.

Since $u^2\cdot R=0$, then $soc(R)$ has dimension at least 2.
Therefore $u\cdot R\cong R/\ker (l_u)$ has dimension at most 1,
where $l_u: R\to R$ is a left multiplication by $u$.
\end{proof}

We need the following lemma.

\begin{lemma}
\label{xxlem3.8}
Let ${\text{char}}\; \Bbbk=2$.
\begin{enumerate}
\item[(1)]
Suppose $A$ is generated by $x_1,x_2,y$ and subject to the relations
$$\begin{aligned}
0&= x_1^2 +c_1 x_2 y+ d_1 y x_2+ e_1 y^2,\\
0&=x_1 x_2+x_2 x_1,\\
0&=x_2^2 +c_3 x_2 y+ d_3 yx_2+ e_3 y^2.
\end{aligned}
$$
Then $A$ is a noetherian Koszul AS regular algebra of global dimension
3 if and only if the parameters satisfy the following conditions
$$d_1=c_1, \; d_3=c_3, \; 
c_3^2-e_3\neq 0,\; c_1e_3-c_3e_1\neq 0.$$
\item[(2)]
Suppose $A$ is generated by $x_1,x_2,y$ and subject to the relations
$$\begin{aligned}
0&=y  x_2-q x_2  y,\\
0&=x_1  x_1 +a x_1  y+
d y^2+e x_2^2,\\
0&= x_1  x_2+x_2  x_1+ a x_2  y.
\end{aligned}
$$
Then $A$ is a noetherian Koszul AS regular algebra of global dimension
3 if and only if the parameters satisfy the following conditions.
$$a=0, q=1, d\neq 0.$$
In this case we can assume that $d=1$.
\item[(3)]
Suppose $A$ is generated by $x_1,x_2,y$ and subject to the relations
$$\begin{aligned}
0&=y  x_2-q x_2  y +c_1 y^2,\\
0&=x_1  x_1 +c x_2  y+
d y^2+e x_2^2,\\
0&= x_1  x_2+x_2  x_1.
\end{aligned}
$$
Then $A$ is a noetherian Koszul AS regular algebra of global dimension
3 if and only if the parameters satisfy the following conditions
$$q=1, d\neq 0, c_1 c=0.$$
As a consequence, we may assume $d=1$, and $(c_1,c)=(1,0)$ or $(0,1)$
by a change of basis.
\end{enumerate}
\end{lemma}

To save some space, we omit the proof of Lemma \ref{xxlem3.8} as it takes
a few pages. Here are a list of ideas used in the proof.

\begin{enumerate}
\item[(a)]
We use the fact that the Hilbert series of $A$ must be $(1-t)^{-3}$ if 
$A$ is noetherian Koszul AS regular of global dimension three.  We will
resolve the overlap ambiguitiy of relations (ideas from Bergman's Diamond
lemma \cite{Be}) to make sure that $\dim A_3=10$. The condition
$\dim A_3=10$ forces some constraints on the parameters.
\item[(b)]
We study the Koszul dual of the algebra. Note that $A$
is Koszul AS regular if and only if the Koszul dual of
$A$ is Koszul and Frobenius, see \cite[Proposition 5.10]{Sm} or
\cite[Corollary D]{LPWZ}. Both Koszul and Frobenius
properties put more constraints on the parameters.
\item[(c)]
We use ideas in \cite[Theorem 4.2]{CV} to show that
$A$ is noetherian Koszul AS regular.
\end{enumerate}

\begin{lemma}
\label{xxlem3.9}
Retain the hypothesis in Lemma {\rm{\ref{xxlem3.7}}}.
Recycle the letter $W$ for $y\otimes V+V\otimes y$.
\begin{enumerate}
\item[(1)]
$R\cap (y\otimes V)=R\cap (V\otimes y)=0$ and
$\dim_{\Bbbk} (R\cap W)\leq 2$.
\item[(2)]
If $\dim_{\Bbbk} (R\cap W)=2$, then $y$ is a normal element.
\item[(3)]
Suppose $\dim_{\Bbbk} (R\cap W)=1$. Then one of the following occurs.
\begin{enumerate}
\item[(3a)]
$j=i+2$ and the relations of $T$ are
$$
y x_2- x_2 y=0,\quad
x_1 x_2- x_2 x_1 +c_2 x_2 y+ d_2 y^2=0,\quad
x_2^2 + y x_1-x_1y +d_3 y^2=0
$$
where $c_2\neq 0$ or $d_3\neq 0$ only if $p=2$ and $d_2\neq 0$
only if $p=3$.
\item[(3b)]
$p=2$, $i=j$ and the relations of $T$ are
$$
yx_2+ x_2y=0,\quad
x_1^2 + y^2+e x_2^2=0,\quad
x_1 x_2+x_2 x_1=0,
$$
where $e\in \Bbbk$.
\item[(3c)]
$p=2$, $i\neq j$ and the relations of $T$ are
$$
y x_2+ x_2 y +b y^2=0,\quad
x_1^2 +c x_2 y+ y^2+e x_2^2,\quad
x_1 x_2+x_2 x_1=0,
$$
with $(b,c)$ being $(0,1)$ or $(1,0)$.
\end{enumerate}
\item[(4)]
If $R$ contains three linearly
independent elements of the form
$$\begin{aligned}
f_1 &= x_1\otimes x_2 + \xi_1,\\
f_2&= x_2\otimes x_1 +\xi_2,\\
f_3&= x_2\otimes x_2 +\xi_3,
\end{aligned}
$$
where $\xi_1,\xi_2,\xi_3\in W$, then $i+1=j$, and, up
to a change of basis, the above relations are
$$x_1x_2 + x_1 y- y x_1=0,\quad
x_2x_1 + x_1 y- y x_1=0,\quad
x_2^2 + x_2y- yx_2=0.
$$
\item[(5)]
Suppose $R\cap W=0$ and $R$ is not in the situation of
part (4). Then $p=2$, $i\neq j$, and $R$ is of the following cases:
\begin{enumerate}
\item[(5a)]
$c_1\in \Bbbk$,
$$
x_1^2+c_1(x_2y+yx_2)+y^2=0,\quad
x_1x_2+x_2x_1=0,\quad
x_2^2+x_2y+yx_2=0.
$$
\item[(5b)]
$e_1\in \Bbbk$,
$$
x_1^2+ x_2y+yx_2+e_1 y^2=0,\quad
x_1x_2+x_2x_1=0,\quad
x_2^2+y^2=0.
$$
\end{enumerate}
\end{enumerate}
\end{lemma}

\begin{proof}
(1) It follows from the proof of Lemma \ref{xxlem3.5}(1)
by replacing $x_3$ by $y$.

(2) The first paragraph of the proof of Lemma
\ref{xxlem3.5}(5) works well for this situation
after replacing $x_3$ by $y$.

(3)
Let $f$ be a basis element in $R\cap W$ and write it as
$$f=a y\otimes x_1 + b y \otimes x_2 + v\otimes y $$
for some $v\in V$.
Suppose that $a\neq 0$. Then $u\cdot f= a y \otimes x_2
+w\otimes y $ for some $w\in V$. So $\{f,u\cdot f\}$
are linearly independent elements in $R\cap W$, yielding a
contradiction. Therefore $a=0$, and in this case we may assume
that $b=1$ and
$$f=y \otimes x_2+ a_1 x_1\otimes y - q x_2\otimes y  +c_1
y \otimes y $$
for some scalars $a_1, q, c_1$. Note that
$$u\cdot f= a_1 x_2\otimes y$$
which must be zero as $\dim_{\Bbbk} R\cap W=1$. Therefore $a_1=0$
and we obtain the first relation of $T$,
$$f_1=y \otimes x_2-q x_2 \otimes y + c_1 y\otimes y$$ where
$c_1\neq 0$ only if $i+1=j$.

Now let $g\in R\setminus W$ and write it as
$$g=a_{11} x_1\otimes x_1 +a_{12} x_1\otimes x_2+ a_{21}
x_2\otimes x_1+a_{22}x_2\otimes x_2 +\phi_0$$
where $\phi_0\in W$.
By a computation, $u^2 \cdot g= 2 a_{11} x_2\otimes x_2+\phi_1$
where $\phi_1\in W$. By Lemma \ref{xxlem3.7}, $u^2\cdot g=0$.
Thus $2a_{11}=0$. Then we have the following two cases to consider.

Case 1: $a_{11}=0$. In this case
$u\cdot g=(a_{12}+a_{21}) x_2\otimes x_2 +\phi_2$ where $\phi_2\in W$.
Since $\dim_{\Bbbk} R/(R\cap W)=2$, either $a_{12}$ or $a_{21}$ is nonzero.
We may assume $a_{12}=1$ by symmetry. If $a_{12}+a_{21}\neq 0$ ,
we have two other relations
$$\begin{aligned}
f_2&=x_1\otimes x_2+a_{21}x_2\otimes x_1 +a x_1\otimes y
     + b y\otimes x_1 +c x_2 \otimes y+ d y\otimes y,\\
(1+a_{21})^{-1} &u\cdot f_2=f_3=x_2\otimes x_2 +(1+a_{21})^{-1} a x_2\otimes y
+(1+a_{21})^{-1} b y\otimes x_2.
\end{aligned}
$$
The subalgebra of $T$ generated by $x_2$ and $y$ subject to relations
$f_1$ and $f_3$ is not a domain. This contradicts the fact that $T$
is a domain. Therefore $a_{12}+a_{21}=0$, or equivalently,
$a_{12}=1$ and $a_{21}=-1$.  Now we have three relations of the form
$$\begin{aligned}
f_1&=y \otimes x_2-q x_2 \otimes y + c_1 y\otimes y,\\
f_2&=x_1\otimes x_2- x_2\otimes x_1 +a_2 x_1\otimes y
     + b_2 y\otimes x_1 +c_2 x_2 \otimes y+ d_2 y\otimes y,\\
f_3&=x_2\otimes x_2 + a_3 x_1\otimes y
     + b_3 y\otimes x_1 +c_3 x_2 \otimes y+ d_3 y\otimes y.
\end{aligned}
$$
Since $T$ is a domain, it may not have two relations only involving
$x_2$ and $y$. Thus $a_3 x_1\otimes y+b_3 y\otimes x_1\neq 0$.
This implies that $u\cdot f_3=a_3 x_2\otimes y+b_3 y\otimes x_2\neq 0$.
Therefore $u\cdot f_3$ and $f_1$ are linearly dependent. Replacing
$y$ by $b_3 y$, we may assume that $b_3=1$. Then $a_3=-q$ and $c_1=0$.
Similarly, we can get $a_2=-q b_2$. After rearranging, using the
fact that $R$ is a $U$-module, we have $j=i+2$ and
$$\begin{aligned}
f_1&=y \otimes x_2-q x_2 \otimes y,\\
f_2&=x_1\otimes x_2- x_2\otimes x_1 +c_2 x_2 \otimes y+ d_2 y\otimes y,\\
f_3&=x_2\otimes x_2 + (y \otimes x_1-q x_1\otimes y)
+d_3 y\otimes y
\end{aligned}
$$
where $c_2\neq 0$ or $d_3\neq 0$ only if $p=2$ and $d_2\neq 0$
only if $p=3$. Finally by resolving the overlap ambiguity of
$(x_1x_2)y=x_1(x_2y)$ with order $y<x_2<x_1$ (details are
omitted, but similar to one given in the proof of Lemma \ref{xxlem3.5}(4)),
we obtain that $q=1$. So we obtain part (3a). In this case it is easy to
see that $T$ is an Ore extension $\Bbbk[x_2,y][x_1;\delta]$.

Case 2: $a_{11}\neq 0$ (so we may assume that $a_{11}=1$) and $p=2$.
Let $f_2=(2i+1-w)\cdot g$ and $f_3=u\cdot f_2$, we have
$$\begin{aligned}
f_2&=x_1\otimes x_1 +a x_1\otimes y +b y\otimes x_1 +c x_2\otimes y+
d y\otimes y+e x_2\otimes x_2,\\
f_3&= x_1\otimes x_2+x_2\otimes x_1+ a x_2\otimes y+ b y\otimes x_2,
\end{aligned}
$$
where $a\neq 0$ or $b\neq 0$ only if $i=j$. Replacing
$x_1$ by $x_1+by$, we can assume that $b=0$. So we have three
relations
$$\begin{aligned}
f_1&=y\otimes x_2-q x_2\otimes y +c_1 y\otimes y,\\
f_2&=x_1\otimes x_1 +a x_1\otimes y +c x_2\otimes y+
d y\otimes y+e x_2\otimes x_2,\\
f_3&= x_1\otimes x_2+x_2\otimes x_1+ a x_2\otimes y,
\end{aligned}
$$
where $c_1\neq 0$ only if $i\neq j$ and $a\neq 0$ only if $i=j$.

If $i=j$, then $c_1=0=c$. In this case we are exactly in the situation of
Lemma \ref{xxlem3.8}(2). By Lemma \ref{xxlem3.8}(2),
$a=0$, $q=1$ and $d\neq 0$ (and we can assume $d=1$ by changing a basis
element). Therefore we obtain (3b) by setting
$d=1$.

If $i\neq j$, then $a=0$. In this case we are exactly in the situation of
Lemma \ref{xxlem3.8}(3). By Lemma \ref{xxlem3.8}(3), $q=1, d\neq 0, c_1c=0$.
By setting $d=1$ and renaming $c_1$ to $b$ and changing
basis elements if necessary , we obtain (3c).

(4) Writing $\xi_i$ out
explicitly, we have the following three linearly
independent elements in $R$
$$\begin{aligned}
f_1 &= x_1\otimes x_2 + a_1 x_1\otimes y+ b_1 y\otimes x_1
       +c_1 x_2\otimes y+ d_1 y\otimes x_2+e_1 y\otimes y,\\
f_2&= x_2\otimes x_1 +a_2 x_1\otimes y+ b_2 y\otimes x_1
       +c_2 x_2\otimes y+ d_2 y\otimes x_2+e_2 y\otimes y,\\
f_3&= x_2\otimes x_2 +a_3 x_1\otimes y+ b_3 y\otimes x_1
       +c_3 x_2\otimes y+ d_3 y\otimes x_2+e_3 y\otimes y.
\end{aligned}
$$

Using the fact that $R\cap W=0$, it is easy to see that
\begin{equation}
\label{E3.9.1}\tag{E3.9.1}
u\cdot f_1=f_3, \quad u\cdot f_2=f_3, \quad u\cdot f_3=0.
\end{equation}
Similarly, we have
\begin{equation}
\notag
(2i+1-w)\cdot f_1= (2i+1-w)\cdot f_2= (2i+2-w)\cdot f_3=0.
\end{equation}
By \eqref{E3.9.1}, we obtain that
$a_1=a_2=c_3$, $b_1=b_2=d_3$, $a_3=b_3=e_3=0$.
By part (1), both $c_3$ and $d_3$ are nonzero. By using
$(2i+2-w)\cdot f_3=0$, we obtain that $i+1=j$. Under
this assumption, $(2i+1-w)\cdot f_1= (2i+1-w)\cdot f_2=0$
imply that
$$c_1=d_1=e_1=0=c_2=d_2=e_2.$$
Therefore we have the following relations, after
setting $a=a_1$ and $b=b_1$,
$$\begin{aligned}
f_1 &= x_1\otimes x_2 + a x_1\otimes y+ b y\otimes x_1,\\
f_2&= x_2\otimes x_1 +a x_1\otimes y+ b y\otimes x_1,\\
f_3&= x_2\otimes x_2 +a x_2\otimes y+ b y\otimes x_2,
\end{aligned}
$$
where $j=i+1$. Using these relations, one can check that
$T$ contains a subalgebra $B:=\Bbbk[x_2][y;\sigma_1,\delta]$
and $T=\sum_{n\geq 0} B x_1^n=\sum_{n\geq 0} x_1^n B$.
Then $T$ is AS regular if and only if $T$ is an iterated Ore
extension of the form
$\Bbbk[x_2][y;\sigma_1,\delta_1][x_1;\sigma_2]$
for some $\sigma_1$, $\sigma_2$ and $\delta_1$ if and
only if $b=-a\neq 0$ (details are omitted).
The assertion follows by setting new $y$ to be $a y$.

(5) Since $R$ is not in the situation of part (4), $R$ contains
an element $g=x_1\otimes x_1+ {\text{other terms}}$.
Then $u^2\cdot g=2(x_2\otimes x_2)\neq 0$. If $p>2$, then
$u^2\cdot g\neq 0$, contradicts Lemma \ref{xxlem3.7}.
Therefore $p=2$.

For the rest of the proof, $p=2$. We have an injective
$U$-morphism
$$R':=(R+W)/W\to (V\otimes V)/W \cong M(2,i)\otimes M(2,i),$$
so we can consider $R'$ as a submodule of $M(2,i)\otimes M(2,i)$.
By Lemma \ref{xxlem3.6}, $R'$ contains elements
$z_3:=-x_1\otimes x_2+x_2\otimes x_1$
and $z_4=x_2\otimes x_2$. Since $R'$ has dimension three,
it must has basis elements $x_1\otimes x_1$, $z_3$ and $z_4$.
Lifting these elements from $R'$ to $R$, we obtain three
linearly independent elements in $R$ (using the fact $p=2$):
$$\begin{aligned}
f_1&= x_1\otimes x_1 + a_1 x_1\otimes y +b_1 y\otimes x_1 +c_1
x_2\otimes y+ d_1 y\otimes x_2+ e_1 y\otimes y,\\
f_2&=x_1\otimes x_2+x_2\otimes x_1
+ a_2 x_1\otimes y +b_2 y\otimes x_1 +c_2
x_2\otimes y+ d_2 y\otimes x_2+ e_2 y\otimes y,\\
f_3&=x_2\otimes x_2+ a_3 x_1\otimes y +b_3 y\otimes x_1 +c_3
x_2\otimes y+ d_3 y\otimes x_2+ e_3 y\otimes y.
\end{aligned}
$$
Under this setting, $R\cap W=0$ implies that
\begin{equation}
\label{E3.9.2}\tag{E3.9.2}
u\cdot f_1=f_2,\quad u\cdot f_2=0, \quad u\cdot f_3=0.
\end{equation}
Further we have,
\begin{equation}
\label{E3.9.3}\tag{E3.9.3}
(2i-w)\cdot f_1=(2i+1-w)\cdot f_2=(2i+2-w)\cdot f_3=0.
\end{equation}
Using \eqref{E3.9.2}-\eqref{E3.9.3}, one has
$$a_2=b_2=e_2=a_3=b_3=0, a_1=c_2, b_1=d_2,$$
and
$$(i-j)a_1=(i-j)b_1=(i-j+1)c_1=(i-j+1)d_1=(i-j+1)c_3=(i-j+1)d_3=0.$$
If $i=j$, then
$$\begin{aligned}
f_1&= x_1\otimes x_1 + a_1 x_1\otimes y +b_1 y\otimes x_1 + e_1 y\otimes y,\\
f_2&=x_1\otimes x_2+x_2\otimes x_1+a_1
x_2\otimes y+ b_1 y\otimes x_2,\\
f_3&=x_2\otimes x_2+ e_3 y\otimes y.
\end{aligned}
$$
Since $f_1$ and $f_3$ are not of the form
$v\otimes w$ for some $v,w\in V$,
$e_3\neq 0$ (so we can assume that $e_3=1$) and
$e_1\neq a_1b_1$. Replacing $x_1$ by $x_1+b_1 y$
(which will not change the $U$-module structure
of $V$), we may assume that $b_1=0$; and up to a rescaling,
$e_1=1$. Thus we have
$i=j$ and
$$\begin{aligned}
f_1&= x_1\otimes x_1 +a x_1\otimes y + y\otimes y,\\
f_2&=x_1\otimes x_2+x_2\otimes x_1+ a x_2\otimes y,\\
f_3&=x_2\otimes x_2+ y\otimes y.
\end{aligned}
$$
We claim that this is not AS regular. To see this,
consider its Koszul dual $B$, which is generated by $y_1:=x_1^{\ast},
y_2:=x_2^{\ast},y_3:=y^{\ast}$ subject to relations
$$y_1^2+y_2^2+y_3^2=0,\quad  ay_1^2+y_1y_3=0, \quad y_3y_1=0,$$
and
$$y_1y_2+y_2y_1=0, \quad ay_1y_2+y_2y_3=0, \quad y_3 y_2=0.$$
Now it is easy to check that $y_3^2 (\Bbbk y_1+\Bbbk y_2+\Bbbk y_3)=0$ which
implies that $B$ is not of finite dimensional and Frobenius.
By \cite[Proposition 5.10]{Sm} or \cite[Corollary D]{LPWZ},
$T$ is not AS regular, yielding a contradiction.

If $i\neq j$, then
$$\begin{aligned}
f_1&= x_1\otimes x_1 +c_1
x_2\otimes y+ d_1 y\otimes x_2+ e_1 y\otimes y,\\
f_2&=x_1\otimes x_2+x_2\otimes x_1,\\
f_3&=x_2\otimes x_2 +c_3
x_2\otimes y+ d_3 y\otimes x_2+ e_3 y\otimes y.
\end{aligned}
$$
By Lemma \ref{xxlem3.8}(1), we have
$$d_1=c_1, \; d_3=c_3, \; 
c_3^2-e_3\neq 0,\; c_1e_3-c_3e_1\neq 0.$$
If $e_3=0$, then $e_1\neq 0$ and $c_3\neq 0$. Up
to a change of basis, the relations of $T$ are
$$\begin{aligned}
x_1^2+c_1(x_2y+yx_2)+y^2&=0,\\
x_1x_2+x_2x_1&=0,\\
x_2^2+x_2y+yx_2&=0.
\end{aligned}
$$
If $e_3\neq 0$, then, up to a change of basis
(by letting new $y$ be $\alpha x_1+\beta y$ for
some scalars $\alpha,\beta$),
we may assume that $e_3=1$ and $c_3=0$, then the relations of
$T$ are
$$\begin{aligned}
x_1^2+ x_2y+yx_2+e_1 y^2&=0,\\
x_1x_2+x_2x_1&=0,\\
x_2^2+y^2&=0.
\end{aligned}
$$

This finishes the proof.
\end{proof}

\begin{lemma}
\label{xxlem3.10}
Retain the hypothesis in Lemma {\rm{\ref{xxlem3.7}}}. If
$y\in S_j\subseteq V$ is a normal element in $T$, then
one of the following holds.
\begin{enumerate}
\item[(1)]
$(i-j)(2i+1-2j)\neq 0$ in ${\mathbb Z}_p$, and the relations
in $T$ are
$$yx_1+ a x_1 y=0, \quad
yx_2+a x_2 y=0, \quad x_2x_1-x_1x_2=0$$
for $a\neq 0$.
\item[(2)]
$2i+1-2j=0$ {\rm{(}}then $i\neq j${\rm{)}} in
${\mathbb Z}_p$, and the relations in $T$ are
$$yx_1+ a x_1 y=0, \quad
yx_2+a x_2 y=0, \quad
x_2x_1-x_1x_2+ \epsilon y^2=0$$
where $a\neq 0$, $\epsilon=0$ or $1$ and $\epsilon (a^2-1)=0$.
\item[(3)]
$i=j$ {\rm{(}}then $2i+1-2j\neq 0${\rm{)}} in ${\mathbb Z}_p$,
and the relations in $T$ are
$$yx_1+ a x_1 y+ b y^2=0, \quad
yx_2+a x_2 y=0,\quad
x_2x_1-x_1x_2+ \epsilon x_2 y$$
where $a\neq 0$, $\epsilon=0$ or $1$ and  $(a+1)(b-\epsilon)=0$.
\end{enumerate}
\end{lemma}

\begin{proof} First of all, every algebra on the list can be
written as an iterated Ore extensions of the form
$\Bbbk[x_2][y;\sigma_1][x_1,\sigma_1,\delta_2]$.
So these are noetherian Koszul AS regular.

Since $y$ is normal and $\Bbbk y$ is a
left $U$-module, the $U$-action on $T$ induces naturally
a $U$-action on $Z:=T/(y)$ where $Z$ is a noetherian
Koszul AS regular algebra of global dimension two.
Note that the degree 1 piece of $Z$ is $Z_1=M(2,i)$,
whence the $U$-action on $Z$ is inner-faithful. By
Proposition \ref{xxpro3.2}(1), $Z$ is commutative,
so $x_2 x_1-x_1 x_2=0$ in $Z$.

Let $R$ be the relation space of $T$. By the previous
paragraph, one can show that $R$ has a basis elements
of the form
$$\begin{aligned}
r_1 &= y x_1 + a_1 x_1 y +b_1 x_2 y+ c_1 y^2,\\
r_2 &= y x_2 + a_2 x_1 y +b_2 x_2 y+ c_2 y^2,\\
r_3 &= x_2 x_1-x_1 x_2 + a_3 x_1 y + b_3 x_2 y +c_3 y^2.
\end{aligned}
$$
By the $U$-action on the basis elements and the fact $\Delta(u)=
u\otimes 1+1\otimes u$, we have
$$\begin{aligned}
u\cdot r_1 &= y x_2 + a_1 x_2 y,\\
u\cdot r_2 &= a_2 x_2 y,\\
u\cdot r_3 &= a_3 x_2 y.
\end{aligned}
$$
Since $u\cdot R\subseteq R$, we obtain that
$$a_2=a_3=c_2=0, \quad a_1=b_2.$$
Similarly, using the above equations and easy computations, we have
$$\begin{aligned}
(i+j-w)\cdot r_1 &= -b_1 x_2 y+ (i-j) c_1 y^2,\\
(i+j+1-w) \cdot r_2 &= 0,\\
(2i+1 -w) \cdot r_3 &= (i-j) b_3 x_2 y +(2i+1-2j)c_3 y^2.
\end{aligned}
$$
Therefore
$$ b_1=0$$ and
\begin{align}
\label{E3.10.1}\tag{E3.10.1}
(i-j) c_1&=0,\\
\label{E3.10.2}\tag{E3.10.2}
(i-j) b_3&=0,\\
\label{E3.10.3}\tag{E3.10.3}
(2i+1-2j) c_{3}&=0.
\end{align}
Now we have three relations of the form
$$\begin{aligned}
r_1 &= y x_1 + a_1 x_1 y + c_1 y^2,\\
r_2 &= y x_2 +a_1  x_2 y,\\
r_3 &= x_2 x_1-x_1 x_2+ b_3 x_2 y +c_3 y^2.
\end{aligned}
$$
with coefficients satisfying \eqref{E3.10.1}-\eqref{E3.10.3}.

Similar to the process of resolving the overlap ambiguity,
we calculate
$$
\begin{aligned}
y(x_2x_1)&= y[x_1 x_2- b_3 x_2 y -c_3 y^2]\\
&=(y x_1) x_2- b_3(y x_2) y -c_3 y^3\\
&=(-a_1 x_1 y - c_1 y^2) x_2 -b_3 (-a_1 x_2 y)y -c_3 y^3\\
&=a_1^2 x_1 x_2 y- a_1^2 c_1 x_2 y^2 +a_1 b_3 x_2 y^2-c_3 y^3\\
(yx_2) x_1&=-a_1 (x_2 y)x_1=-a_1x_2 (y x_1)\\
&=-a_1 x_2 [-a_1 x_1 y - c_1 y^2]\\
&=a_1^2 (x_2 x_1) y +a_1 c_1 x_2 y^2\\
&=a_1^2 [x_1 x_2- b_3 x_2 y -c_3 y^2] y +a_1 c_1 x_2 y^2\\
&=a_1^2 x_1 x_2y - a_1^2 b_3 x_2 y^2 +a_1 c_1 x_2 y^2-a_1^2 c_3 y^3.
\end{aligned}
$$
Since $\{x_1^ix_2^jy^k \mid i,j,k\geq\}$ is a $\Bbbk$-linear
basis, we have
\begin{align}
\label{E3.10.4}\tag{E3.10.4}
c_3 (a_1^2-1)&=0\\
\label{E3.10.5}\tag{E3.10.5}
-a_1^2 c_1 +a_1 b_3&=
-a_1^2 b_3 +a_1 c_1.
\end{align}
Since $R$ does not contain an element of the form
$v\otimes w$ for some $v,w\in V$, $a_1\neq 0$.
Now the system of equations \eqref{E3.10.1}-\eqref{E3.10.5}
has the following solutions:
\begin{enumerate}
\item[(a)]
$(i-j)(2i+1-2j)\neq 0$ in ${\mathbb Z}_p$,
$r_1=yx_1+ a x_1 y$, $r_2=yx_2+a x_2 y$, $r_3=x_2x_1-x_1x_2$.
This is case (1).
\item[(b)]
$2i+1-2j=0$ (then $i\neq j$) in ${\mathbb Z}_p$,
$r_1=yx_1+ a x_1 y$, $r_2=yx_2+a x_2 y$, $r_3=x_2x_1-x_1x_2+ cy^2$
where $c(a^2-1)=0$. Replacing $y$ by $\sqrt{c}y$, we obtain the case
(2).
\item[(c)]
$i=j$ (then $2i+1-2j\neq 0$) in ${\mathbb Z}_p$,
$r_1=yx_1+ a x_1 y+cy^2$, $r_2=yx_2+a x_2 y$, $r_3=x_2x_1-x_1x_2+ b x_2 y$
where $(a+1)(b-c)=0$.
This is case (3) after rescaling.
\end{enumerate}
This finishes the proof.
\end{proof}

\begin{proof}[Proof of Theorem \ref{xxthm0.5}]
First of all, it is routine to check that all algebras in
Theorem \ref{xxthm0.5} are noetherian connected graded Koszul
AS regular of global dimension three, as we did in the proofs
of Lemmas \ref{xxlem3.5}, \ref{xxlem3.8}, \ref{xxlem3.9} and
\ref{xxlem3.10}.

If $T$ is commutative, we only need to specify the $U$-action
on $V:=T_1$. There are two cases: either $V=M(3,i)$ or
$V=M(2,i)\oplus S_j$. This is part (1).

Next we suppose that $T$ is not commutative.
If $V=M(3,i)$, then, by Lemma \ref{xxlem3.5},
only part (2) can occur.

For the rest of the proof, we assume that $T$ is
not commutative and $V=M(2,i)\oplus S_j$. We use
$\dim_{\Bbbk} R\cap W$ to classify the pairs
$(U,T)$.

If $\dim_{\Bbbk} R\cap W=2$, by Lemma \ref{xxlem3.9}(2),
$y$ is a normal element. By Lemma \ref{xxlem3.10}, we obtain
parts (4,5,6). If $\dim_{\Bbbk} R\cap W=1$, by
Lemma \ref{xxlem3.9}(3) we obtain parts (7,8,9).
If $\dim_{\Bbbk} R\cap W=0$, by Lemma \ref{xxlem3.9}(4,5),
we obtain parts (3) and (10). This finishes the proof.
\end{proof}

\section{Easy observations}
\label{xxsec4}

From the limited information in the global dimension 2
case, we see differences between the semisimple and
non-semisimple actions.

\begin{observation}
\label{xxobs4.1}
The following remarks demonstrate differences between
semisimple and non-semisimple Hopf actions on AS regular
algebras. Suppose we are in the setting of Proposition
{\rm{\ref{xxpro3.2}}}.
\begin{enumerate}
\item[(1)]
Let $p$ be an odd prime and let $i=\frac{p-1}{2}$.
Let $V=M(2,i)$. Then the relation of $T=\Bbbk[V]$
is of the form $r=x_1\otimes x_2-x_2\otimes x_1
\in V\otimes V$. By definition, $w\cdot x_j=(i+j-1)x_j$
for $j=1,2$. Hence
$$\begin{aligned}
w\cdot r &= (w\otimes 1+1\otimes w) (x_1\otimes x_2-x_2\otimes x_1)\\
&=(2i+1) (x_1\otimes x_2-x_2\otimes x_1) \\
&= p (x_1\otimes x_2-x_2\otimes x_1)=pr =0.
\end{aligned}
$$
It is clear that $u\cdot r=0$. Therefore the $U$-action
on $\Bbbk r$ is trivial. If we use the statement in
\cite[Theorem 2.1]{CKWZ1} as our definition of trivial
homological determinant, then the $U$-action on $T$
has trivial homological determinant. By Proposition
{\rm{\ref{xxpro3.2}(2)}}, $T^U$ is AS regular. Therefore
\cite[Theorem 0.6]{CKWZ1} fails without $H$ being semisimple.
\item[(2)]
Suppose the $U$-action on $T$ has trivial homological determinant
as in part {\rm{(1)}}. Since $T$ is a free module over $T^U$,
$\End_{T^U}(T)$ is a matrix algebra over $T^U$.
So $T\# U$ is not isomorphic to $\End_{T^U}(T)$, consequently,
\cite[Theorem 0.3]{CKWZ2} fails without $H$ being semisimple.
\item[(3)]
Suppose the $U$-action on $T$ has trivial homological determinant
as in part {\rm{(1)}}. Note that ${\rm{CMreg}}(T)=0$ and
$${\rm{CMreg}}(T^U)={\rm{CMreg}}(\Bbbk[x_1^p, x_2^p])=
2p-2\neq 0,$$
where the definition of Castelnuovo-Mumford regularity,
denoted ${\rm{CMreg}}$, can be found
in \cite[Definition 2.9]{KWZ}. This example shows that
\cite[Lemma 2.15(2)]{KWZ} fails without $H$ being semisimple.
\item[(4)]
In Proposition {\rm{\ref{xxpro3.2}(2)}}, $T^U$ is AS regular.
We are wondering if there is a version of the Shephard-Todd-Chevalley
Theorem for non-semisimple Hopf actions on AS regular algebras, 
even if $U$ does not have any nontrivial grouplike element. In the 
case of Proposition {\rm{\ref{xxpro3.2}(2a)}}, we have that 
$\dim U=p^2$ and that the product of degrees of generators is 
$p$. Hence the conclusion of both \cite[Proposition 1.8(4)]{KWZ} 
and \cite[Conjecture 0.3]{FKMW1} fails. Since $U$ is not 
semisimple, we are wondering if $U$ is still qualified to be called a 
{\it reflection Hopf algebra}, see \cite[Definition 3.2]{KKZ2}.
\item[(5)]
When $H$ is semisimple {\rm{(}}with mild hypotheses{\rm{)}},
by \cite[Corollary 3.10]{CFM}, the rank of $T$ as a
left $T^H$-module is equal to $\dim_{\Bbbk} H$. In Proposition
{\rm{\ref{xxpro3.2}(2a)}}, the rank of $T$ as a left $T^U$-module
is $p$, while $\dim_{\Bbbk} U=p^2$. Therefore 
\cite[Corollary 3.10]{CFM} fails without $H$ being semisimple.
\item[(6)]
When $H$ is semisimple, acting on an AS regular algebra $T$
inner-faithfully, $T$ is usually a left free $H$-module. In the
commutative case see, for example, \cite[Theorem 6.19(2,3)]{OT}.
In the noncommutative case, this was verified for many examples,
see \cite{FKMW1, FKMW2}. However, in the non-semisimple case,
$T$ is never a free $H$-module as ${_H T_0}\cong {_H\Bbbk}$ cannot 
be projective.
\item[(7)]
Consider the case when $p=2$ and $i=1$ in Proposition
{\rm{\ref{xxpro3.2}(2a)}}. Then it is clear that $T_0=\Bbbk\cong S_0$,
$T_1=\Bbbk x_1+\Bbbk x_2\cong M(2,1)$, and
$T_2=(\Bbbk x_1+\Bbbk x_2)x_2+\Bbbk x_1^2\cong M(2,1)\oplus S_0$.
For $i\geq 3$, 
$T_i\cong \begin{cases} M(2,1)^{m} & i=2m,\\
M(2,1)^{m}\oplus S_0 & i=2m+1.\end{cases}$
Therefore neither $S_1$ nor $M(2,0)$ appears as a direct
summand of $T$. Further $\ann_U(T)=\Bbbk wu\neq 0$.
So the $U$-action on $T$ is not faithful, though it is
inner-faithful. However, in the semisimple case, an
inner-faithful $H$-action on an AS regular algebra
is expected to be faithful.
\end{enumerate}
\end{observation}

Next we verify the claim made in Remark \ref{xxrem0.7}(3).

\begin{lemma}
\label{xxlem4.2}
Let $H$ be a Hopf algebra containing $K$ as a Hopf subalgebra.
Let $H$ act on an algebra $T$ inner-faithfully.
\begin{enumerate}
\item[(1)]
Suppose every nonzero Hopf ideal of $K$ contains a nonzero
skew primitive element. Then the induced $K$-action on $T$ is
inner-faithful.
\item[(2)]
If $K$ is pointed, then the induced $K$-action on $T$ is
inner-faithful.
\item[(3)]
If $K=U$, then the induced $U$-action on $T$ is
inner-faithful.
\end{enumerate}
\end{lemma}

\begin{proof} (1)
If the $K$-action on $T$ is not inner-faithful,
then $x\cdot T=0$ for some nonzero skew primitive
in $K$. It is clear that $x$ is also a skew primitive
element in $H$. Let $I$ be the ideal of $H$ generated by $x$.
Since $x$ is a primitive element, $I$ is a Hopf ideal of $H$.
It is clear that $I\cdot T=0$ as $x\cdot T=0$. Therefore the
$H$-action is not inner-faithful. The assertion follows.

(2) By \cite[Corollary 5.4.7]{Mo}, every nonzero Hopf
ideal of the pointed Hopf algebra $K$ contains a nonzero
skew primitive. The assertion follows from part (1).

(3) This is a special case of (2) as $U$ is pointed
(in fact, connected).
\end{proof}

\begin{proof}[Proof of Proposition \ref{xxpro0.4}]
Let $H$ act inner-faithfully on a noetherian Koszul AS regular
algebra $T$ of global dimension two. By Lemma
\ref{xxlem4.2}, the induced $U$-action on $T$ is inner-faithful.
By Proposition \ref{xxpro3.2}(1), $T$ is $\Bbbk[x_1,x_2]$.
\end{proof}

Let $m$ be an integer $\geq 2$. Let $\{p_{ij} \mid 1\leq i< j\leq m\}$
be a set of nonzero scalars. Define $p_{ij}=p_{ji}^{-1}$ if
$i>j$. Recall that the skew polynomial ring is defined to be
$$\Bbbk_{p_{ij}}[x_1,\cdots,x_m]=
\frac{\Bbbk\langle x_1,\cdots,x_m\rangle}{(x_j x_i=p_{ij} x_i x_j, \;
\forall \;i<j)}.$$

\begin{proposition}
\label{xxpro4.3}
Suppose $p_{ij}\neq 1$ for all $i<j$. Let
$T=\Bbbk_{p_{ij}}[x_1,\cdots,x_m]$ and let $H$ be a Hopf algebra
containing $\Bbbk[u]/(u^p)$, with $u$ being primitive, as a Hopf
subalgebra. Then there is no inner-faithful
homogeneous $H$-action on $T$.
\end{proposition}

\begin{proof} By Lemma \ref{xxlem4.2}(2), we may assume $H=\Bbbk[u]/(u^p)$
with $u$ being primitive. Suppose to the contrary that there is an
inner-faithful homogeneous $H$-action on $T$. By \cite[Lemma 5.9(d)]{KKZ1}
and cocommutativity of $H$, $H$ acts on the Koszul dual (equal to
the Ext-algebra) of $T$, denoted by
$$B:={\text{Ext}}^\ast_{T}(\Bbbk,\Bbbk)=
\frac{\Bbbk\langle y_1,\cdots,y_m\rangle}{(y_j y_i=-p_{ij}^{-1} y_i y_j, \;
\forall \;i<j, y_i^2, \; \forall \; i)}$$
where $y_i=x_i^{\ast}$ for each $i$.

Since $u$ is primitive,  $u$ acts on $B$ as a derivation.
For every $i$, write $u(y_i)=\sum_{j} a_{ji} y_j$.
Then
$$\begin{aligned}
0&=u (y_i^2) =(\sum_{j} a_{ji} y_j) y_i +
y_i(\sum_{j} a_{ji} y_j)\\
&=\sum_{j} a_{ji}(y_j y_i+y_i y_j)
=\sum_{j\neq i} a_{ji}(y_j y_i+y_i y_j)\\
&=\sum_{j\neq i} a_{ji}(1-p_{ij}^{-1})y_i y_j
\end{aligned}
$$
which implies that $a_{ji}=0$ for all $j\neq i$.
Equivalently, $u(y_i)=a_{ii} y_i$ for each $i$.
Since $u^p=0$, we obtain that $a_{ii}^p=0$, or
equivalently, $a_{ii}=0$. Thus $u \cdot B_1=0$. Recall that
$B_1=(T_1)^{\ast}$. Hence $u\cdot T_1=0$.
Consequently, $u\cdot T=0$ and the $H$-action is
not inner-faithful, yielding a contradiction.
\end{proof}

As an immediately consequence of Lemma \ref{xxlem4.2}(2) and
Proposition \ref{xxpro4.3}, we have the following very simple 
universal non-existence result.

\begin{corollary}
\label{xxcor4.4}
Let $H$ be a nontrivial finite dimensional connected local Hopf algebra.
Then there is no inner-faithful homogeneous $H$-action on 
$T=\Bbbk_{p_{ij}}[x_1,\cdots,x_m]$ where $p_{ij}\neq 1$ for all $i<j$.
\end{corollary}

\begin{proof} Note that every nontrivial finite dimensional connected 
local Hopf algebra contains $\Bbbk[u]/(u^p)$. Therefore the assertion
follows from Lemma \ref{xxlem4.2}(2) and Proposition \ref{xxpro4.3}.
\end{proof}

It is easy to check that $\Bbbk[u]/(u^p)$ acts inner-faithfully
on both the polynomial ring $\Bbbk[x_1,x_2]$ and the Jordan 
plane $\Bbbk\langle x_1,x_2\rangle/(x_1x_2-x_2x_1-x_1^2)$. 

Note that most of $p^3$-dimensional connected Hopf algebras
in the classification \cite{NWW1, NWW2, Wa1, Wa2} contain a
Hopf subalgebra of the form $\Bbbk[u]/(u^p)$. So
Proposition \ref{xxpro4.3} applies to these Hopf algebras.
The next observation is a continuation of Remark \ref{xxrem0.7}(7).

\begin{observation}
\label{xxobs4.5}
Let $H$ be a Hopf algebra containing $U$ as a Hopf subalgebra.
If $H$ acts on a noetherian
Koszul AS regular algebra $T$ of global dimension three, then
it induces naturally a $U$-action on $T$. This induced $U$-action
must be inner-faithful by Lemma {\rm{\ref{xxlem4.2}(3)}}. As a
consequence, $T$ must be one of the algebras listed in Theorem
{\rm{\ref{xxthm0.5}}}. This basically gives a proof of Corollary 
{\rm{\ref{xxcor0.6}}}.
\begin{enumerate}
\item[(1)]
Theorem {\rm{\ref{xxthm0.5}}} is helpful for understanding explicit
$H$-actions on $T$ when $H$ contains $U$. Even if $H$ is of wild
representation type, we can start from the list of $T$ in Theorem
{\rm{\ref{xxthm0.5}}} to work out all possible $H$-actions on $T$.
This strategy is different from the one in the proof of
Theorem {\rm{\ref{xxthm0.5}}}.
\item[(2)]
Suggested by Zhuang's result \cite[Theorem 1.1]{Zh}, every pointed 
Hopf algebra over a field of positive characteristic is expected 
to contain one of the special Hopf algebras such as $U$, $U_0$, or 
Taft algebras and so on {\rm{(}}this list should be short{\rm{)}}. 
By understanding actions on noetherian Koszul AS regular algebras 
of global dimension three under Hopf algebras from this list, we 
should get a pretty good picture of Hopf actions on Koszul AS 
regular algebras of global dimension three for all pointed Hopf
algebra over a field of positive characteristic.
\item[(3)]
The classification of $p^3$-dimensional connected or pointed
Hopf algebras is undergoing in \cite{NWW1, NWW2, Wa1, Wa2}.
It is known from their work that $U$ appears as a Hopf subalgebra
in many of their examples. Therefore Theorem {\rm{\ref{xxthm0.5}}}
is really helpful for understanding the actions on noetherian Koszul
AS regular algebras of global dimension three under the Hopf
algebras listed in \cite{NWW1, NWW2, Wa1, Wa2}.
\end{enumerate}
\end{observation}

\section{Comments, projects, and remarks}
\label{xxsec5}

In this final section we randomly collect some general comments, related
projects, remarks and questions related to the projects in the previous
sections.

As noted in \cite[Remark 3.13(2)]{CVZ}, the representations
of the Drinfeld double of a Hopf algebra $H$ are generally much
more complicated than the representations of $H$. So we start
with the following question.

\begin{question}
\label{xxque5.1}
What is the Green ring of the Drinfeld double of $U$?
\end{question}

The work \cite{Ch2, Ch3} is closely related to this
question.

The next few projects concern the Hopf actions
of AS regular algebras of global dimension three or higher.
A test case is the next.

\begin{project}
\label{xxpro5.2}
Classify all $U$-actions on noetherian Koszul AS regular algebras
of global dimension 4.
\end{project}

The tensor decomposition in Theorem \ref{xxthm0.2}
is again the key to this project.

\begin{project}
\label{xxpro5.3}
For all $U$-actions given in Theorem \ref{xxthm0.5},
work out the fixed subrings $T^U$ and understand the connection
between the $U$-actions and the properties of $T^U$.
\end{project}

This would be a very interesting project as little is known
about non-semisimple Hopf actions on AS regular
algebras.

As noted before Taft algebras are similar to $U$ in several
aspects. So the following project seems doable.

\begin{project}
\label{xxpro5.4}
Let $H$ be a Taft algebra $H_n(q)$ of dimension $n^2$ over
a field of arbitrary characteristic \cite[p. 767]{CVZ}. Classify
all inner-faithful $H$-actions on noetherian Koszul AS
regular algebras of global dimension three.
\end{project}

Unrelated to Hopf actions, the results in Section \ref{xxsec2} are
useful for the computation of Frobenius-Perron dimensions of
representations of $U$.

The Frobenius-Perron dimension of an object in a semisimple
finite tensor (or fusion) category was introduced by
Etingof-Nikshych-Ostrik in 2005 \cite{ENO}. Since then it has become
an extremely useful invariant in the study of fusion categories
and representations of semisimple (weak and/or quasi-)Hopf
algebras. Recently, a new definition of Frobenius-Perron
dimension was introduced in \cite{CGW1, CGW2} where the
original definition was extended from an object in a semisimple
finite tensor category to an endofunctor of any $\Bbbk$-linear
category. In particular, it is defined for objects in
non-semisimple $\Bbbk$-linear monoidal categories. This new
Frobenius-Perron dimension has been computed in various cases, see
\cite{CGW1, CGW2, Xu, ZWD, ZZ}.

\begin{remark}
\label{xxrem5.5}
\begin{enumerate}
\item[(1)]
Xu computed Frobenius-Perron dimensions of representations
of Taft algebras in \cite{Xu}, whose method can be used to
give Frobenius-Perron dimension of representations of $U$
for some small prime $p$. For example,
\begin{enumerate}
\item[(a)]
let $p=2$ and $X$ be a finite dimensional $U$-module, then
$\fpdim (X)$ equals to the $\Bbbk$-dimension of $X$;
\item[(b)]
let $p=3$ and $X=\sum_{i=0}^2 \sum_{l=1}^3 M(l,i)^{\oplus a_{li}}$,
then
$$\fpdim (X)=
\frac{1}{2}\left[(\alpha+\gamma)+\sqrt{(\alpha-\gamma)^2+4 \beta^2}\right]
$$
where
$$\begin{aligned}
\alpha&=a_{10}+a_{11}+a_{12}+2(a_{20}+a_{21}+a_{22})+3(a_{30}+a_{31}+a_{32}),\\
\beta &=a_{12}+a_{21}+a_{22}+a_{30}+a_{31}+a_{32},\\
\gamma&=a_{10}+a_{22}+a_{31}.
\end{aligned}
$$
This formula is similar to \cite[Proposition 7.1]{Xu}. Since 
our Convention {\rm{\ref{xxcon1.4}}} is slightly different from 
the one used in \cite{Xu}, the formula does not match up 
with \cite[Proposition 7.1]{Xu} exactly.
\end{enumerate}
\item[(2)]
It would be interesting to work out a formula of $\fpdim(X)$
when $p=5$.
\end{enumerate}
\end{remark}

A general algebra $B$ is usually of wild representation type,
then it is impossible to understand all indecomposable
left $B$-modules. Sometime it is possible to work out all brick
modules which are a special class of indecomposable module. A left
$B$-module is called a {\it brick} if $\Hom_{B}(M,M)=\Bbbk$. Brick
modules are fundamental objects in the study of Frobenius-Perron
dimension of endofunctors \cite{CGW1, ZZ}. Even when $B$ is of wild
representation type, it would be extremely helpful to understand all
brick modules. The next remark is a consequence of Proposition 
\ref{xxpro1.7}.

\begin{remark}
\label{xxrem5.6}
Let $H$ be a 2-step iterated Hopf Ore extension given in
\cite{BZ}. In parts {\rm{(2,3)}},
Let $Z$ be the center of $H$ and let $M$ be a brick left $H$-module.
Let $\fm=\ann_H(M)$.
\begin{enumerate}
\item[(1)]
Suppose $H$ is commutative. Then each brick left $H$-module is 1-dimensional
and there is a one-to-one correspondence between brick
$H$-modules and a closed point in ${\text{Spec}}\; H$.
\item[(2)]
Suppose $H$ is noncommutative and $d_0=0$ as in \cite[Proposition 8.2]{BZ}.
Then
$$M\cong \begin{cases} {\text{the unique $1$-dimensional simple associated to $\fm$}} &
{\text{if $\fm$ is not Azumaya,}}\\
{\text{the unique $p$-dimensional simple associated to $\fm$}} &
{\text{if $\fm$ is Azumaya.}}
\end{cases}
$$
\item[(3)]
Suppose $H$ is noncommutative and $d_0\neq 0$ as in \cite[Proposition 8.2]{BZ}.
Then
$$M\cong \begin{cases} {\text{one of $p^2$ indecomopsable modules associated to $\fm$}} &
{\text{if $\fm$ is not Azumaya,}}\\
{\text{the unique $p$-dimensional simple associated to $\fm$}} &
{\text{if $\fm$ is Azumaya.}}
\end{cases}
$$
\end{enumerate}
\end{remark}

\subsection*{Acknowledgments}
The authors thank Ellen Kirkman for useful conversations on the 
subject and for carefully reading an earlier version of this 
paper, and thank Yongjun Xu for sharing his unpublished notes 
\cite{Xu} and several interesting ideas. 
H.-X. Chen and D.-G. Wang thank J.J. Zhang
and the Department of Mathematics at University of
Washington for its hospitality during their visits.
H.-X. Chen was partially supported by the National Natural
Science Foundation of China (Nos. 11571298 and 11971418).
D.-G. Wang was partially supported by the National Natural
Science Foundation of China  (No. 11871301) and
the NSF of Shandong Province (No. ZR2019MA060).
J.J. Zhang was partially supported by the US National Science
Foundation (Nos. DMS-1700825 and DMS-200101).

\end{document}